\documentclass[11pt]{amsart}

\newtheorem{thm}{Theorem}[section]
\newtheorem*{thm*}{Theorem}

\newtheorem{lem}[thm]{Lemma}
\newtheorem*{lem*}{Lemma}
\newtheorem{mainthm}{Theorem}
\newtheorem*{mainthm*}{Theorem}
\newtheorem{maincor}[mainthm]{Corollary}
\newtheorem{prop}[thm]{Proposition}

\theoremstyle{definition}

\newtheorem*{case*}{Case}
\newtheorem{conj}[thm]{Conjecture}

\newtheorem{defn}[thm]{Definition}
\newtheorem*{defn*}{Definition}

\newtheorem*{exmp*}{Example}

\renewcommand{\thestep}{}

\theoremstyle{remark}

\renewcommand{\thecase}{}

\newtheorem{rmk}[thm]{Remark}
\newtheorem*{rmk*}{Remark}

\makeatletter
\def\alphenumi{
  \def\theenumi{\alph{enumi}}
  \def\p@enumi{\theenumi}
  \def\labelenumi{(\@alph\c@enumi)}}
\makeatother

\makeatletter
\def\thecase{\@arabic\c@case}
\makeatother

\makeatletter
\def\thestep{\@arabic\c@step}
\makeatother

\newcount\hh
\newcount\mm
\mm=\time
\hh=\time
\divide\hh by 60
\divide\mm by 60
\multiply\mm by 60
\mm=-\mm
\advance\mm by \time
\def\hhmm{\number\hh:\ifnum\mm<10{}0\fi\number\mm}

\let\oldmarginpar\marginpar
\renewcommand\marginpar[1]{\-\oldmarginpar[\raggedleft\footnotesize #1]%
{\raggedright\footnotesize #1}}

\newcommand\CC{\mathbb{C}}

\newcommand\EE{\mathbb{E}}

\newcommand\NN{\mathbb{N}}

\newcommand\RR{\mathbb{R}}

\newcommand\ZZ{\mathbb{Z}}

\newcommand\fg{{\mathfrak{g}}}

\newcommand\sB{{\mathscr{B}}}
\newcommand\sC{{\mathscr{C}}}
\newcommand\sD{{\mathscr{D}}}
\newcommand\sE{{\mathscr{E}}}
\newcommand\sF{{\mathscr{F}}}
\newcommand\sG{{\mathscr{G}}}
\newcommand\sH{{\mathscr{H}}}
\newcommand\sI{{\mathscr{I}}}

\newcommand\sL{{\mathscr{L}}}
\newcommand\sM{{\mathscr{M}}}
\newcommand\sN{{\mathscr{N}}}

\newcommand\sT{{\mathscr{T}}}
\newcommand\sU{{\mathscr{U}}}
\newcommand\sV{{\mathscr{V}}}

\newcommand\sX{{\mathscr{X}}}
\newcommand\sY{{\mathscr{Y}}}

\newcommand\eps{\varepsilon}

\newcommand\U{\operatorname{U}}

\newcommand\less{\setminus}

\newcommand\ad{{\operatorname{ad}}}
\newcommand\Ad{{\operatorname{Ad}}}

\newcommand\Aut{\operatorname{Aut}}

\DeclareMathOperator{\Crit}{Crit}

\newcommand\divg{\operatorname{div}}

\newcommand\End{\operatorname{End}}

\newcommand{\esssup}{\operatornamewithlimits{ess\ sup}}

\newcommand\grad{\operatorname{grad}}

\newcommand\Hom{\operatorname{Hom}}

\newcommand\Ker{\operatorname{Ker}}

\newcommand\Real{\operatorname{Re}}

\newcommand\vol{\operatorname{vol}}

\newcommand\id{{\mathrm{id}}}

\newcommand\mutatis{{\emph{mutatis mutandis }}}

\newcommand\spinc{\text{$\text{spin}^c$ }}

\newcommand\Spinc{\text{$\text{Spin}^c$}}

\numberwithin{equation}{section}

\usepackage{amscd}
\usepackage{amssymb}
\usepackage{graphicx}
\usepackage{hyperref}
\usepackage{mathrsfs}
\usepackage[usenames]{color}
\usepackage{paralist}
\usepackage{slashed}
\usepackage{url}
\usepackage{verbatim}
\usepackage{xr}

\hypersetup{pdftitle={{\L}ojasiewicz--Simon gradient inequalities for analytic and Morse--Bott functions on Banach spaces}}
\hypersetup{pdfauthor={Paul M. N. Feehan and Manousos Maridakis}}

\begin{document}

\title[{\L}ojasiewicz--Simon gradient inequalities]{{\L}ojasiewicz--Simon gradient inequalities for analytic and Morse--Bott functions on Banach spaces}

\author[Paul M. N. Feehan]{Paul M. N. Feehan}
\address{Department of Mathematics, Rutgers, The State University of New Jersey, 110 Frelinghuysen Road, Piscataway, NJ 08854-8019, United States of America}
\email{feehan@math.rutgers.edu}

\author[Manousos Maridakis]{Manousos Maridakis}
\address{Department of Mathematics, Rutgers, The State University of New Jersey, 110 Frelinghuysen Road, Piscataway, NJ 08854-8019, United States of America}
\email{mmaridaki1@gmail.com}

%COMMENT Remove hours and minutes for arXiv and journal versions and fix date
%\date{\today{ }\hhmm}
\date{This version: September 16, 2019, incorporating final galley proof corrections except for publisher's omission of unused equation number labels and reordering of bibliography. \emph{Journal f\"ur die reine und angewandte Mathematik} (2019), https://doi.org/10.1515/crelle-2019-0029}

\begin{abstract}
We prove several abstract versions of the {\L}ojasiewicz--Simon gradient inequality for an analytic function on a Banach space that generalize previous abstract versions of this inequality, weakening their hypotheses and, in particular, that of the well-known infinite-dimensional version of the gradient inequality due to {\L}ojasiewicz \cite{Lojasiewicz_1965} proved by Simon as \cite[Theorem 3]{Simon_1983}. We prove that the optimal exponent of the {\L}ojasiewicz--Simon gradient inequality is obtained when the function is Morse--Bott, improving on similar results due to Chill \cite[Corollary 3.12]{Chill_2003}, \cite[Corollary 4]{Chill_2006}, Haraux and Jendoubi \cite[Theorem 2.1]{Haraux_Jendoubi_2007}, and Simon \cite[Lemma 3.13.1]{Simon_1996}. In \cite{Feehan_Maridakis_Lojasiewicz-Simon_application_harmonic_maps}, we apply our abstract gradient inequalities to prove {\L}ojasiewicz--Simon gradient inequalities for the harmonic map energy function using Sobolev spaces which impose minimal regularity requirements on maps between closed, Riemannian manifolds. Those inequalities generalize those of Kwon \cite[Theorem 4.2]{KwonThesis}, Liu and Yang \cite[Lemma 3.3]{Liu_Yang_2010}, Simon \cite[Theorem 3]{Simon_1983}, \cite[Equation (4.27)]{Simon_1985}, and Topping \cite[Lemma 1]{Topping_1997}. In \cite{Feehan_Maridakis_Lojasiewicz-Simon_coupled_Yang-Mills_v6}, we prove {\L}ojasiewicz--Simon gradient inequalities for coupled Yang--Mills energy functions using Sobolev spaces which impose minimal regularity requirements on pairs of connections and sections. Those inequalities generalize that of the pure Yang--Mills energy function due to the first author \cite[Theorems 23.1 and 23.17]{Feehan_yang_mills_gradient_flow_v4} for base manifolds of arbitrary dimension and due to R\r{a}de \cite[Proposition 7.2]{Rade_1992} for dimensions two and three.
\end{abstract}

% AMS 2010 subject classifications (used in AMS journals)


% AMS keywords (used in AMS journals)

% Acknowledge support
\thanks{Paul Feehan was partially supported by National Science Foundation grant DMS-1510064 and the Oswald Veblen Fund and Fund for Mathematics (Institute for Advanced Study, Princeton) during the preparation of this article.}

\maketitle
%TODO Remove for journal version
\tableofcontents

\section{Introduction}
\label{sec:Introduction}
Since its discovery by {\L}ojasiewicz in the context of analytic functions on Euclidean spaces\footnote{The first page number refers to the version of {\L}ojasiewicz's original manuscript mimeographed by IHES while the page number in parentheses refers to the cited LaTeX version of his manuscript prepared by M. Coste and available on the Internet.} \cite[Proposition 1, p. 92 (67)]{Lojasiewicz_1965} and subsequent generalization by Simon to a class of analytic functions on certain H\"older spaces \cite[Theorem 3]{Simon_1983}, the \emph{{\L}ojasiewicz--Simon gradient inequality} has played a significant role in analyzing questions such as the
\begin{inparaenum}[\itshape a\upshape)]
\item global existence, convergence, and analysis of singularities for solutions to nonlinear evolution equations that are realizable as gradient-like systems,
\item uniqueness of tangent cones, and
\item energy gaps and discreteness of energies.
\end{inparaenum}
For applications of the {\L}ojasiewicz--Simon gradient inequality to gradient flows arising in geometric analysis, beginning with the harmonic map energy function, we refer to Irwin \cite{IrwinThesis}, Kwon \cite{KwonThesis}, Liu and Yang \cite{Liu_Yang_2010}, Simon \cite{Simon_1985}, and Topping \cite{ToppingThesis, Topping_1997}; for gradient flow for the Chern-Simons function, see Morgan, Mrowka, and Ruberman \cite{MMR}; for gradient flow for the Yamabe function, see Brendle \cite[Lemma 6.5 and Equation (100)]{Brendle_2005} and Carlotto, Chodosh, and Rubinstein \cite{Carlotto_Chodosh_Rubinstein_2015}; for Yang-Mills gradient flow, we refer to our monograph \cite{Feehan_yang_mills_gradient_flow_v4}, R\r{a}de \cite{Rade_1992}, and Yang \cite{Yang_2003aim}; for mean curvature flow, we refer to the survey by Colding and Minicozzi \cite{Colding_Minicozzi_2014sdg}; and for Ricci curvature flow, see Ache \cite{Ache_2011arxiv}, Haslhofer \cite{Haslhofer_2012cvpde}, Haslhofer and M{\"u}ller \cite{Haslhofer_Muller_2014}, and Kr{\"o}ncke \cite{Kroncke_2015cvpde, Kroncke_2013arxiv}.  In Feehan and Maridakis \cite{Feehan_Maridakis_Lojasiewicz-Simon_application_harmonic_maps}, we apply Theorems \ref{mainthm:Lojasiewicz-Simon_gradient_inequality} and \ref{mainthm:Lojasiewicz-Simon_gradient_inequality2} to prove {\L}ojasiewicz--Simon gradient inequalities for the harmonic map energy function (see Theorem \ref{mainthm:Lojasiewicz-Simon_gradient_inequality_energy_functional_Riemannian_manifolds} and Corollary \ref{maincor:Lojasiewicz-Simon_gradient_inequality_energy_functional_Riemannian_manifolds_L2}).

For applications of the {\L}ojasiewicz--Simon gradient inequality to proofs of global existence, convergence, convergence rate, and stability of non-linear evolution equations arising in other areas of mathematical physics (including the Cahn-Hilliard, Ginzburg-Landau, Kirchoff-Carrier, porous medium, reaction-diffusion, and semi-linear heat and wave equations), we refer to the monograph by Huang \cite{Huang_2006} for a comprehensive introduction and to the articles by Chill \cite{Chill_2003, Chill_2006}, Chill and Fiorenza \cite{Chill_Fiorenza_2006}, Chill, Haraux, and Jendoubi \cite{Chill_Haraux_Jendoubi_2009}, Chill and Jendoubi \cite{Chill_Jendoubi_2003, Chill_Jendoubi_2007}, Feireisl and Simondon \cite{Feireisl_Simondon_2000}, Feireisl and Tak{\'a}{\v{c}} \cite{Feireisl_Takac_2001}, Grasselli, Wu, and Zheng \cite{Grasselli_Wu_Zheng_2009}, Haraux \cite{Haraux_2012}, Haraux and Jendoubi \cite{Haraux_Jendoubi_1998, Haraux_Jendoubi_2007, Haraux_Jendoubi_2011}, Haraux, Jendoubi, and Kavian \cite{Haraux_Jendoubi_Kavian_2003}, Huang and Tak{\'a}{\v{c}} \cite{Huang_Takac_2001}, Jendoubi \cite{Jendoubi_1998jfa}, Rybka and Hoffmann \cite{Rybka_Hoffmann_1998, Rybka_Hoffmann_1999}, Simon \cite{Simon_1983}, and Tak{\'a}{\v{c}} \cite{Takac_2000}. For applications to fluid dynamics, see the articles by Feireisl, Lauren{\c{c}}ot, and Petzeltov{\'a} \cite{Feireisl_Laurencot_Petzeltova_2007}, Frigeri, Grasselli, and Krej{\v{c}}{\'{\i}} \cite{Frigeri_Grasselli_Krejcic_2013}, Grasselli and Wu \cite{Grasselli_Wu_2013}, and Wu and Xu \cite{Wu_Xu_2013}.

For applications of the {\L}ojasiewicz--Simon gradient inequality to proofs of energy gaps, we refer to our article \cite{Feehan_yangmillsenergygapflat}. A key feature of our versions of the {\L}ojasiewicz--Simon gradient inequality for the pure Yang-Mills energy function \cite[Theorems 23.1 and 23.17]{Feehan_yang_mills_gradient_flow_v4} is that they hold for $W^{2,p}$ and $W^{1,2}$ Sobolev norms, which are considerably weaker than the $C^{2,\alpha}$ H{\"o}lder norms originally employed by Simon in \cite[Theorem 3]{Simon_1983} and afford greater flexibility in applications. For example, when $(X,g)$ is a closed, four-dimensional, Riemannian manifold, the $W^{1,2}$ Sobolev norm on (bundle-valued) one-forms is (in a suitable sense) \emph{quasi-conformally invariant} with respect to conformal changes in the Riemannian metric $g$. In Feehan and Maridakis \cite{Feehan_Maridakis_Lojasiewicz-Simon_coupled_Yang-Mills_v6}, we apply Theorems \ref{mainthm:Lojasiewicz-Simon_gradient_inequality} and \ref{mainthm:Lojasiewicz-Simon_gradient_inequality2} to prove {\L}ojasiewicz--Simon gradient inequalities for coupled Yang-Mills energy functions (see Theorems \ref{mainthm:Lojasiewicz-Simon_gradient_inequality_boson_Yang--Mills_energy_function} and \ref{mainthm:Lojasiewicz-Simon_gradient_inequality_fermion_Yang--Mills_energy_function}).

There are essentially three approaches to establishing a {\L}ojasiewicz--Simon gradient inequality for a particular energy function arising in geometric analysis or mathematical physics:
\begin{inparaenum}[\itshape 1\upshape)]
\item establish the inequality from first principles,
\item adapt the argument employed by Simon in the proof of his
\cite[Theorem 3]{Simon_1983}, or
\item apply an abstract version of the {\L}ojasiewicz--Simon gradient inequality for an analytic or Morse--Bott function on a Banach space.
\end{inparaenum}
Most famously, the first approach is exactly that employed by Simon in \cite{Simon_1983}, although this is also the avenue followed by Kwon \cite{KwonThesis}, Liu and Yang \cite{Liu_Yang_2010} and Topping \cite{ToppingThesis, Topping_1997} for the harmonic map energy function and by R\r{a}de \cite{Rade_1992} for the Yang-Mills energy function. Occasionally a development from first principles may be necessary, as discussed by Colding and Minicozzi in \cite{Colding_Minicozzi_2014sdg}. However, in almost all of the remaining examples cited, one can derive a {\L}ojasiewicz--Simon gradient inequality for a specific application from an abstract version for an analytic or Morse--Bott function on a Banach space. For this strategy to work well, one desires an abstract {\L}ojasiewicz--Simon gradient inequality with the weakest possible hypotheses and a proof of such a gradient inequality (Theorem \ref{mainthm:Lojasiewicz-Simon_gradient_inequality}) is one purpose of the present article. We also prove an abstract {\L}ojasiewicz--Simon gradient inequality, with the optimal exponent, for a Morse--Bott function on a Banach space, generalizing and unifying previous versions of the {\L}ojasiewicz--Simon gradient inequality with optimal exponent obtained in specific examples.

While our abstract versions of the {\L}ojasiewicz--Simon gradient inequality (Theorems \ref{mainthm:Lojasiewicz-Simon_gradient_inequality_dualspace}, \ref{mainthm:Lojasiewicz-Simon_gradient_inequality}, \ref{mainthm:Lojasiewicz-Simon_gradient_inequality2}, and \ref{mainthm:Optimal_Lojasiewicz-Simon_gradient_inequality_Morse-Bott_energy_functional}) are versatile enough to apply to many problems in geometric analysis, mathematical physics, and applied mathematics, it is worth noting that there are situations where it appears difficult to derive a {\L}ojasiewicz--Simon gradient inequality for a specific application from an abstract version. For example, a gradient inequality due to Feireisl, Issard-Roch, and Petzeltov{\'a} applies to functions that are not $C^2$ \cite[Proposition 4.1 and Remark 4.1]{Feireisl_Issard-Roch_Petzeltova_2004}. Colding and Minicozzi describe certain gradient inequalities \cite[Theorems 2.10 and 2.12]{Colding_Minicozzi_2014sdg} employed in their work on mean curvature flow that do not appear to follow from abstract {\L}ojasiewicz--Simon gradient inequalities or even the usual arguments underlying their proofs \cite[Section 1]{Colding_Minicozzi_2014sdg}. Nevertheless, those examples should not preclude consideration of abstract {\L}ojasiewicz--Simon gradient inequalities with the broadest possible application.

In the remainder of our Introduction, we summarize the principal results of our article, beginning with two versions of the abstract {\L}ojasiewicz--Simon gradient inequality for analytic functions on Banach spaces in Section \ref{subsec:Lojasiewicz-Simon_gradient_inequality_abstract_functional} and a version of the abstract {\L}ojasiewicz--Simon gradient inequality for Morse--Bott functions on Hilbert spaces in Section \ref{subsec:Lojasiewicz-Simon_gradient_inequality_abstract_functional_Morse-Bott}. We illustrate applications of the preceding theorems with statements of {\L}ojasiewicz--Simon gradient inequalities for the harmonic map energy function in Section \ref{subsec:Lojasiewicz-Simon_gradient_inequality_harmonic_map_functional} and for the coupled Yang-Mills boson and fermion energy functions in Section \ref{subsec:Lojasiewicz-Simon_gradient_inequality_boson_fermion_Yang--Mills_function}.

\subsection{{\L}ojasiewicz--Simon gradient inequalities for analytic functions on Banach spaces}
\label{subsec:Lojasiewicz-Simon_gradient_inequality_abstract_functional}
We begin the subsection with two abstract versions of Simon's infinite-dimensional
version \cite[Theorem 3]{Simon_1983} of the {\L}ojasiewicz gradient
inequality \cite{Lojasiewicz_1965}. A slightly less general form of Theorem \ref{mainthm:Lojasiewicz-Simon_gradient_inequality_dualspace} (see Remark \ref{rmk:Embedding_hypothesis_Huang_theorem_2-4-5}) is stated by Huang as
\cite[Theorem 2.4.5]{Huang_2006} but no proof was given and it does
not follow directly from his \cite[Theorem 2.4.2(i)]{Huang_2006} (see Feehan and Maridakis \cite[Theorem E.2]{Feehan_Maridakis_Lojasiewicz-Simon_coupled_Yang-Mills_v6}). Huang cites \cite[Proposition 3.3]{Huang_Takac_2001} for the proof of his \cite[Theorem 2.4.5]{Huang_2006} but the hypotheses of \cite[Proposition 3.3]{Huang_Takac_2001} assume that $\sX$ is a Hilbert space.
The proof of Theorem \ref{mainthm:Lojasiewicz-Simon_gradient_inequality} that we include in Section \ref{sec:Lojasiewicz-Simon_gradient_inequality_analytic_and_Morse-Bott_energy_functionals} generalizes that of Feireisl and Tak{\'a}{\v{c}} for their \cite[Proposition 6.1]{Feireisl_Takac_2001} in the case of the Ginzburg-Landau energy function.

Let $\sX$ be a Banach space and let $\sX^*$ denote its continuous dual space. We call a bilinear form\footnote{Unless stated otherwise, all Banach spaces are considered to be real in this article.}, $b:\sX\times\sX \to \RR$, \emph{definite} if $b(x,x) \neq 0$ for all $x \in \sX\less\{0\}$. We say that a continuous \emph{embedding} of a Banach space into its continuous dual space, $\jmath:\sX\to\sX^*$, is \emph{definite} if the pullback of the canonical pairing, $\sX\times\sX \ni (x,y) \mapsto \langle x,\jmath(y)\rangle_{\sX\times\sX^*} \to \RR$, is a definite bilinear form.

\begin{mainthm}[{\L}ojasiewicz--Simon gradient inequality for analytic functions on Banach spaces]
\label{mainthm:Lojasiewicz-Simon_gradient_inequality_dualspace}
Let $\sX \subset \sX^*$ be a continuous, definite embedding of a Banach space into its dual space. Let $\sU \subset \sX$ be an open subset, let $\sE:\sU\to\RR$ be an analytic function, and let $x_\infty\in\sU$ be a critical point of $\sE$, that is, $\sE'(x_\infty) = 0$. Assume that $\sE''(x_\infty):\sX\to \sX^*$ is a Fredholm operator with index zero. Then there are constants $Z \in (0, \infty)$ and $\sigma \in (0,1]$ and $\theta \in [1/2,1)$ with the following significance. If $x \in \sU$ obeys
\begin{equation}
\label{eq:Lojasiewicz-Simon_gradient_inequality_neighborhood}
\|x-x_\infty\|_\sX < \sigma,
\end{equation}
then
\begin{equation}
\label{eq:Lojasiewicz-Simon_gradient_inequality_analytic_functional}
\|\sE'(x)\|_{\sX^*} \geq Z|\sE(x) - \sE(x_\infty)|^\theta.
\end{equation}
\end{mainthm}

\begin{rmk}[Comments on the embedding hypothesis in Theorem \ref{mainthm:Lojasiewicz-Simon_gradient_inequality_dualspace} and comparison with Huang's Theorem]
\label{rmk:Embedding_hypothesis_Huang_theorem_2-4-5}
The hypothesis in Theorem
\ref{mainthm:Lojasiewicz-Simon_gradient_inequality_dualspace} on the
definiteness of the continuous embedding $\sX \subset \sX^*$ is easily achieved given a
continuous and dense embedding $\eps:\sX\hookrightarrow\sH$ of $\sX$ into a Hilbert space
$\sH$. Indeed, since $\langle y,\jmath(x)\rangle_{\sX\times\sX^*} =
(\eps(y),\eps(x))_\sH$ for all $x,y\in\sX$, then $\langle
x,\jmath(x)\rangle_{\sX\times\sX^*} = 0$ implies $x = 0$. The adjoint map $\eps^*:\sH^* \hookrightarrow \sX^*$ is also a continuous embedding and so the composition $\sX \subset \sH \cong \sH^*\subset \sX^*$ (where the isometric isomorphism $\sH \cong \sH^*$ is given by the Riesz map) yields the desired definite embedding $\sX\subset \sX^*$. See
\cite[Remark 3, page 136]{Brezis} or \cite[Lemma
D.1]{Feehan_Maridakis_Lojasiewicz-Simon_harmonic_maps_v5} for additional
details.
Theorem \ref{mainthm:Lojasiewicz-Simon_gradient_inequality_dualspace} generalizes Huang's \cite[Theorem 2.4.5]{Huang_2006} by replacing his hypothesis that there is a Hilbert space $\sH$ such that $\sX \subset \sH \subset \sX^*$ is a sequence of continuous embeddings with our stated hypothesis on the embedding of $\sX$.
\end{rmk}

\begin{rmk}[Index of a Fredholm Hessian operator on a reflexive Banach space]
\label{rmk:Index_Fredholm_Hessian_operator_reflexive_Banach_space}
If $\sX$ is a reflexive Banach space in Theorem \ref{mainthm:Lojasiewicz-Simon_gradient_inequality_dualspace}, then the hypothesis that $\sE''(x_\infty):\sX\to \sX^*$ has index zero can be omitted, since $\sE''(x_\infty)$ is always a symmetric operator and thus necessarily has index zero when $\sX$ is reflexive by \cite[Lemma D.3]{Feehan_Maridakis_Lojasiewicz-Simon_harmonic_maps_v5}.
\end{rmk}

\begin{rmk}[Replacement of Hilbert by Banach space dual norms in {\L}ojasiewicz--Simon gradient inequalities]
\label{rmk:Replace_Hilbert_by_Banach space dual norm}
The structure of the original result due to Simon \cite[Theorem
3]{Simon_1983} was simplified in certain applications by Feireisl and
Simondon \cite[Proposition 6.1]{Feireisl_Simondon_2000}, R\r{a}de
\cite[Proposition 7.2]{Rade_1992}, Rybka and Hoffmann \cite[Theorem
3.2]{Rybka_Hoffmann_1998}, \cite[Theorem 3.2]{Rybka_Hoffmann_1999},
and Tak{\'a}{\v{c}} \cite[Proposition 8.1]{Takac_2000} by replacing
the $L^2(M;V)$ norm used by Simon in his \cite[Theorem 3]{Simon_1983}
with dual Sobolev norms, such as $W^{-1,2}(M;V)$, and replacing the
$C^{2,\alpha}(M;V)$ H\"older norm used by Simon to define the
neighborhood of the critical point with a Sobolev $W^{1,2}(M;V)$ norm,
where $M$ is a closed Riemannian manifold and $V$ is a Riemannian
vector bundle equipped with a compatible connection. The choice $\sX =
W^{1,2}(M;V)$ in Theorem
\ref{mainthm:Lojasiewicz-Simon_gradient_inequality_dualspace} is very
convenient, but imposes constraints on the dimension of $M$ and
nonlinearity of the derivative map. The difficulties are explained
further in
\cite{Feehan_Maridakis_Lojasiewicz-Simon_coupled_Yang-Mills_v6} and
Remark~\ref{rmk:Choice_Banach_and_Hilbert_spaces_Lojasiewicz-Simon_gradient_inequality}.
%COMMENT PF5-6-2017: There are two issues - Banach algebra property
%and Fredholmness. The space W^{1,2}(M;V) cannot be a Banach algebra
%unless V is an algebra too.
\end{rmk}

As emphasized by an anonymous referee, the hypotheses of Theorem \ref{mainthm:Lojasiewicz-Simon_gradient_inequality_dualspace} are restrictive. For example, even though its hypotheses allow $\sX$ to be a Banach space, when the Hessian $\sE''(x_\infty)$ is defined by an elliptic, linear, second-order partial differential operator, then (in the notation of Remark \ref{rmk:Choice_Banach_and_Hilbert_spaces_Lojasiewicz-Simon_gradient_inequality}) one is naturally led to choose $\sX$ to be the Hilbert space $W^{1,2}(M;V)$ with dual space $\sX^* = W^{-1,2}(M;V^*)$ in order to obtain the required Fredholm property. However, such a choice could make it impossible to simultaneously obtain the required real analyticity of the function $\sE:\sX \supset \sU \to \RR$. As explained in Remark \ref{rmk:Choice_Banach_and_Hilbert_spaces_Lojasiewicz-Simon_gradient_inequality}, the forthcoming generalization greatly relaxes these constraints and implies Theorem \ref{mainthm:Lojasiewicz-Simon_gradient_inequality_dualspace} as a corollary. We first recall the concept of a gradient map \cite[Section 2.1B]{Huang_2006}, \cite[Section 2.5]{Berger_1977}.

\begin{defn}[Gradient map]
\label{defn:Huang_2-1-1}
(See \cite[Definition 2.1.1]{Huang_2006}.)
Let $\sU\subset \sX$ be an open subset of a Banach space $\sX$ and let $\tilde\sX$ be a Banach space with continuous embedding $\tilde\sX \subseteqq \sX^*$. A continuous map $\sM:\sU\to \tilde\sX$ is called a \emph{gradient map} if there exists a $C^1$ function $\sE:\sU\to\RR$ such that
\begin{equation}
\label{eq:Differential_and_gradient_maps}
\sE'(x)v = \langle v,\sM(x)\rangle_{\sX\times\sX^*}, \quad \forall\, x \in \sU, \quad v \in \sX,
\end{equation}
where $\langle \cdot , \cdot \rangle_{\sX\times\sX^*}$ is the canonical bilinear form on $\sX\times\sX^*$. The real-valued function $\sE$ is called a \emph{potential} for the gradient map, $\sM$.
\end{defn}

When $\tilde\sX = \sX^*$ in Definition \ref{defn:Huang_2-1-1}, then the derivative and gradient maps coincide.

\begin{mainthm}[Refined {\L}ojasiewicz--Simon gradient inequality for analytic functions on Banach spaces]
\label{mainthm:Lojasiewicz-Simon_gradient_inequality}
Let $\sX$ and $\tilde\sX$ be Banach spaces with continuous embeddings
%TODO - Where do we need $\sX \subset \tilde\sX?
\[
  \sX \subset \tilde\sX \subset \sX^*
\]
and such that the embedding $\sX \subset \sX^*$ is definite. Let $\sU \subset \sX$ be an open subset, $\sE:\sU\to\RR$ be a $C^2$ function with real analytic gradient map $\sM:\sU\to\tilde\sX$, and $x_\infty\in\sU$ be a critical point of $\sE$, that is, $\sM(x_\infty) = 0$. If $\sM'(x_\infty):\sX\to \tilde\sX$ is a Fredholm operator with index zero, then there are constants $Z \in (0,\infty)$ and $\sigma \in (0,1]$ and $\theta \in [1/2, 1)$ with the following significance. If $x \in \sU$ obeys
\begin{equation}
\label{eq:Lojasiewicz-Simon_gradient_inequality_neighborhood_general}
\|x-x_\infty\|_\sX < \sigma,
\end{equation}
then
\begin{equation}
\label{eq:Lojasiewicz-Simon_gradient_inequality_analytic_functional_general}
\|\sM(x)\|_{\tilde\sX} \geq Z|\sE(x) - \sE(x_\infty)|^\theta.
\end{equation}
\end{mainthm}

\begin{rmk}[Comments on the embedding hypothesis in Theorem \ref{mainthm:Lojasiewicz-Simon_gradient_inequality}]
\label{rmk:Embedding_hypothesis_Lojasiewicz-Simon_gradient_inequality}
As explained in Remark \ref{rmk:Embedding_hypothesis_Huang_theorem_2-4-5}, the hypothesis in Theorem
\ref{mainthm:Lojasiewicz-Simon_gradient_inequality} on the
definiteness of the continuous embedding $\sX \subset \sX^*$ is easily achieved given a
continuous and dense embedding of $\sX$ into a Hilbert space $\sH$. The composition of the continuous embeddings $\sX\subset \tilde\sX\subset \sX^*$ in Theorem \ref{mainthm:Lojasiewicz-Simon_gradient_inequality} induces the same embedding $\sX\subset \sX^*$ as that constructed in Remark \ref{rmk:Embedding_hypothesis_Huang_theorem_2-4-5}. 
\end{rmk}  

\begin{rmk}[Previous versions of the {\L}ojasiewicz--Simon gradient inequality for analytic functions on Banach spaces]
\label{rmk:Previous_abstract_Lojasiewicz-Simon_gradient_inequalities}
The results \cite[Theorem 3.10 and Corollary 3.11]{Chill_2003} and \cite[Corollary 3]{Chill_2006} due to Chill provide versions of the {\L}ojasiewicz--Simon gradient inequality for a $C^2$ function on a Banach space that overlap with Theorem \ref{mainthm:Lojasiewicz-Simon_gradient_inequality}; see
\cite[Proposition 3.12]{Carlotto_Chodosh_Rubinstein_2015} for a nice exposition of Chill's version \cite{Chill_2003} of the abstract {\L}ojasiewicz--Simon gradient inequality. However, the hypotheses of Theorem  \ref{mainthm:Lojasiewicz-Simon_gradient_inequality} (for analytic gradient maps) and the forthcoming Theorem \ref{mainthm:Optimal_Lojasiewicz-Simon_gradient_inequality_Morse-Bott_energy_functional} (for Morse--Bott functions) are simpler and easier to verify in many applications.

The result \cite[Theorem 4.1]{Haraux_Jendoubi_2011} due to Haraux and Jendoubi is an abstract {\L}ojasiewicz--Simon gradient inequality which they argue is optimal based on examples that they discuss in \cite[Section 3]{Haraux_Jendoubi_2011}. However, while the hypothesis in Theorem \ref{mainthm:Lojasiewicz-Simon_gradient_inequality} is replaced by their alternative requirements that $\Ker\sE''(x_\infty)$ be finite-dimensional and $\sE''(x_\infty)$ obey a certain coercivity condition on the orthogonal complement of $\Ker\sE''(x_\infty)$, they require $\sX$ to be a Hilbert space.

Theorem \ref{mainthm:Lojasiewicz-Simon_gradient_inequality} also considerably strengthens and simplifies \cite[Theorem 2.4.2(i)]{Huang_2006} (see Feehan and Maridakis \cite[Theorem E.2]{Feehan_Maridakis_Lojasiewicz-Simon_coupled_Yang-Mills_v6}).
\end{rmk}

\begin{rmk}[On the choice of Banach spaces in applications of Theorem \ref{mainthm:Lojasiewicz-Simon_gradient_inequality}]
\label{rmk:Choice_Banach_and_Hilbert_spaces_Lojasiewicz-Simon_gradient_inequality}
The hypotheses of Theorem \ref{mainthm:Lojasiewicz-Simon_gradient_inequality} are designed to give the most flexibility in applications of a {\L}ojasiewicz--Simon gradient inequality to analytic functions on Banach spaces. An example of a convenient choice of Banach spaces modeled as Sobolev spaces, when $\sM'(x_\infty)$ is realized as an elliptic partial differential operator of order $m$, would be
\[
\sX = W^{k,p}(M;V), \quad \tilde\sX = W^{k-m,p}(M;V), \quad\text{and}\quad \sX^* = W^{-k,p'}(M;V),
\]
where $k\in\ZZ$ is an integer, $p \in (1,\infty)$ is a constant with dual H\"older exponent $p'\in(1,\infty)$ defined by $1/p+1/p'=1$, while $M$ is a closed Riemannian manifold of dimension $d\geq 2$ and $V$ is a Riemannian vector bundle with a compatible covariant derivative
\[
  \nabla:C^\infty(M;V) \to C^\infty(M;T^*X\otimes V),
\]
and $W^{k,p}(M;V)$ denote Sobolev spaces defined in the standard way \cite{Aubin_1998}. When the integer $k$ is chosen to be large enough, the verification of analyticity of the gradient map $\sM:\sU\to\tilde\sX$ is straightforward. Normally, that is the case when $k\geq m+1$ and $(k-m)p>d$ or $k-m=d$ and $p=1$, since $W^{k-m,p}(M;\CC)$ is then a Banach algebra by \cite[Theorem 4.39]{AdamsFournier}.
\end{rmk}

Theorem \ref{mainthm:Lojasiewicz-Simon_gradient_inequality} appears to us to be the most widely applicable abstract version of the {\L}ojasiewicz--Simon gradient inequality that we are aware of in the literature. However, for applications where $\sM'(x_\infty)$ is realized as an elliptic partial differential operator of \emph{even} order, $m=2n$, and the nonlinearity of the gradient map is sufficiently mild, it often suffices to choose $\sX$ to be the Banach space $W^{n,2}(M;V)$ and choose $\tilde\sX = \sX^*$ to be the Banach space $W^{-n,2}(M;V)$. The distinction between the derivative $\sE'(x) \in \sX^*$ and the gradient $\sM(x) \in \tilde\sX$ then disappears. Similarly, the distinction between the Hessian $\sE''(x_\infty) \in (\sX\times\sX)^*$ and the Hessian operator $\sM'(x_\infty) \in \sL(\sX,\tilde\sX)$ disappears. Finally, if $\sE:\sX\supset\sU \to \RR$ is real analytic, then the simpler Theorem \ref{mainthm:Lojasiewicz-Simon_gradient_inequality_dualspace} is often adequate for applications.

\subsection{Generalized {\L}ojasiewicz--Simon gradient inequalities for analytic functions on Banach spaces and gradient maps valued in Hilbert spaces}
\label{subsec:Lojasiewicz-Simon_gradient_inequalities_Hilbert_space_introduction}
While Theorem \ref{mainthm:Lojasiewicz-Simon_gradient_inequality} has important applications to proofs of global existence, convergence, convergence rates, and stability of gradient flows defined by an energy function $\sE:\sX\supset \sU \to \RR$ with gradient map $\sM:\sX\supset \sU \to \tilde\sX$ (see \cite[Section 2.1]{Feehan_yang_mills_gradient_flow_v4} for an introduction and Simon \cite{Simon_1983} for his pioneering development), the gradient inequality \eqref{eq:Lojasiewicz-Simon_gradient_inequality_analytic_functional_general} is most useful when it has the form
\[
\|\sM(x)\|_{\sH} \geq Z|\sE(x) - \sE(x_\infty)|^\theta, \quad\forall\, x \in \sU \text{ with } \|x-x_\infty\|_\sX < \sigma,
\]
where $\sH$ is a Hilbert space and the Banach space $\sX$ is a dense subspace of $\sH$ with continuous embedding $\sX \subset \sH$, and so $\sH^* \subset \sX^*$ is also a continuous embedding. Thus, we have $\sX \subset \sH \cong \sH^* \subset \sX^*$ and $(\sX,\sH,\sX^*)$ is\footnote{Though we do not necessarily require $\sX$ to be reflexive.} an ``evolution triple'' (see \cite[Remark 3, p. 136]{Brezis} or \cite[Definition 3.4.3]{Denkowski_Migorski_Papageorgiou_intro_nonlinear_analysis_applications}) and $\sH$ is called the ``pivot space''.

For example, to obtain Theorem \ref{mainthm:Lojasiewicz-Simon_gradient_inequality_energy_functional_Riemannian_manifolds} for the harmonic map energy function, we choose
\[
\sX = W^{k,p}(M;f_\infty^*TN),
\]
but for applications to gradient flow, we would like to replace the gradient inequality \eqref{eq:Lojasiewicz-Simon_gradient_inequality_harmonic_map_energy_functional_Riemannian_manifold} by
\[
\|\sM(f)\|_{L^2(M;f^*TN)}
\geq
Z|\sE(f) - \sE(f_\infty)|^\theta,
\]
under the original {\L}ojasiewicz--Simon neighborhood condition \eqref{eq:Lojasiewicz-Simon_gradient_inequality_harmonic_map_neighborhood_Riemannian_manifold},
\[
\|f - f_\infty\|_{W^{k,p}(M)} < \sigma.
\]
Unfortunately, such an $L^2$ gradient inequality
(or Simon's \cite[Theorem 3]{Simon_1983}, \cite[Equation (4.27)]{Simon_1985}) does not follow from Theorem \ref{mainthm:Lojasiewicz-Simon_gradient_inequality} when $M$ has dimension $d \geq 4$, as explained in the proof of Corollary \ref{maincor:Lojasiewicz-Simon_gradient_inequality_energy_functional_Riemannian_manifolds_L2} and Remark \ref{rmk:Lojasiewicz-Simon_gradient_inequality_energy_functional_manifolds_L2_exclude_d_geq_4}; see also \cite{Feehan_harmonic_map_relative_energy_gap}. However, these $L^2$ gradient inequalities \emph{are} implied by the forthcoming Theorem \ref{mainthm:Lojasiewicz-Simon_gradient_inequality2}, which generalizes and simplifies Huang's \cite[Theorem 2.4.2 (i)]{Huang_2006} (see Feehan and Maridakis \cite[Theorem E.2]{Feehan_Maridakis_Lojasiewicz-Simon_coupled_Yang-Mills_v6}).

\begin{mainthm}[Generalized {\L}ojasiewicz--Simon gradient inequality for analytic functions on Banach spaces]
\label{mainthm:Lojasiewicz-Simon_gradient_inequality2}
Let $\sX$ and $\tilde\sX$ be Banach spaces with continuous embeddings
\[
  \sX \subset \tilde\sX \subset \sX^*
\]
and such that the embedding $\sX \subset \sX^*$ is definite.  Let $\sU \subset \sX$ be an open subset, $\sE:\sU\to\RR$ be an analytic function, and $x_\infty\in\sU$ be a critical point of $\sE$, that is, $\sE'(x_\infty) = 0$. Let
\[
\sX\subset \sG \subset \tilde\sG \quad \text{and} \quad \tilde\sX \subset \tilde\sG \subset \sX^*
\]
be continuous embeddings of Banach spaces such that the compositions
\[
\sX\subset \sG\subset \tilde\sG \quad \text{and}\quad \sX\subset \tilde\sX\subset \tilde\sG,
\]
induce the same embedding, $\sX \subset \tilde\sG$. Let $\sM:\sU\to\tilde\sX$ be a gradient map for $\sE$ in the sense of Definition \ref{defn:Huang_2-1-1}. Suppose that for each $x \in \sU$, the bounded, linear operator
\[
\sM'(x): \sX \to \tilde \sX
\]
has an extension
\[
\sM_1(x): \sG \to \tilde\sG
\]
such that the map
\[
\sU \ni x \mapsto \sM_1(x) \in \sL(\sG, \tilde\sG) \quad\hbox{is continuous}.
\]
If $\sM'(x_\infty):\sX\to \tilde\sX$ and $\sM_1(x_\infty):\sG\to \tilde\sG$ are Fredholm operators with index zero, then there are constants $Z \in (0,\infty)$ and $\sigma \in (0,1]$ and $\theta \in [1/2, 1)$ with the following significance. If $x \in \sU$ obeys
\begin{equation}
\label{eq:Lojasiewicz-Simon_gradient_inequality_neighborhood_general2}
\|x-x_\infty\|_\sX < \sigma,
\end{equation}
then
\begin{equation}
\label{eq:Lojasiewicz-Simon_gradient_inequality_analytic_functional_general2}
\|\sM(x)\|_{\tilde\sG} \geq Z|\sE(x) - \sE(x_\infty)|^\theta.
\end{equation}
\end{mainthm}

Suppose now that $\tilde\sG = \sH$, a Hilbert space, so that the embedding $\sG\subset \sH$  in Theorem \ref{mainthm:Lojasiewicz-Simon_gradient_inequality2}, factors through $\sG\subset \sH \cong \sH^*$ and therefore
\[
\sE'(x)v = \langle v, \sM(x) \rangle_{\sX\times\sX^*} = (v, \sM(x))_\sH, \quad\forall\, x \in \sU \text{ and } v \in \sX,
\]
using the continuous embeddings $\tilde\sX \subset \sH \subset \sX^*$. As we noted in Remark \ref{rmk:Embedding_hypothesis_Huang_theorem_2-4-5}, the hypothesis in Theorem \ref{mainthm:Lojasiewicz-Simon_gradient_inequality2} that the embedding $\sX \subset \sX^*$ is definite is implied by the assumption that $\sX \subset \sH$ is a continuous embedding into a Hilbert space. By Theorem \ref{mainthm:Lojasiewicz-Simon_gradient_inequality2}, if $x \in \sU$ obeys
\begin{equation}
\label{eq:Lojasiewicz-Simon_gradient_inequality_neighborhood_general_Hilbert_space}
\|x-x_\infty\|_\sX < \sigma,
\end{equation}
then
\begin{equation}
\label{eq:Lojasiewicz-Simon_gradient_inequality_analytic_functional_Hilbert_space}
\|\sM(x)\|_{\sH} \geq Z|\sE(x) - \sE(x_\infty)|^\theta,
\end{equation}
as desired.

\begin{rmk}
If the Banach spaces are instead modeled as H\"older spaces, as in Simon \cite{Simon_1983}, a convenient choice of Banach spaces would be
\[
\sX = C^{k,\alpha}(M;V), \quad \tilde\sX = C^{k-m,\alpha}(M;V), \quad\text{and}\quad \sH = L^2(M;V),
\]
where $\alpha \in (0,1)$ and $k\geq m$, and these H\"older spaces are defined in the standard way \cite{Aubin_1998}.
\end{rmk}

\begin{rmk}[Combination of Theorems \ref{mainthm:Lojasiewicz-Simon_gradient_inequality} and \ref{mainthm:Lojasiewicz-Simon_gradient_inequality2}]
\label{rmk:Lojasiewicz-Simon_gradient_inequality_combined}
Naturally, it is possible to combine Theorems \ref{mainthm:Lojasiewicz-Simon_gradient_inequality} and \ref{mainthm:Lojasiewicz-Simon_gradient_inequality2} and their proofs; however, we avoid doing so for reasons of clarity and exposition.
\end{rmk}

\subsection{{\L}ojasiewicz--Simon gradient inequalities for Morse--Bott functions on Banach spaces}
\label{subsec:Lojasiewicz-Simon_gradient_inequality_abstract_functional_Morse-Bott}
It is of considerable interest to know when the optimal exponent $\theta = 1/2$ is achieved, since in that case one can prove (see \cite[Theorem 24.21]{Feehan_yang_mills_gradient_flow_v4}, for example) that a global solution $u:[0,\infty)\to\sX$ to a gradient system governed by the {\L}ojasiewicz--Simon gradient inequality,
\[
\frac{du}{dt} = -\sE'(u(t)), \quad u(0) = u_0,
\]
has \emph{exponential} rather than mere power-law rate of convergence
to a critical point $u_\infty$. One simple version of such an
optimal {\L}ojasiewicz--Simon gradient inequality is provided in Huang
\cite[Proposition 2.7.1]{Huang_2006}. Although interesting, its
hypotheses are very restrictive and the result is a special case of Theorem
\ref{mainthm:Lojasiewicz-Simon_gradient_inequality_dualspace}, where $\sX$ is assumed to be a Hilbert space and the Hessian $\sE''(x_\infty):\sX\to \sX^*$ is an invertible operator. See Haraux, Jendoubi, and Kavian \cite[Proposition 1.1]{Haraux_Jendoubi_Kavian_2003} for a similar result.

For the harmonic map energy function, a more interesting optimal {\L}ojasiewicz--Simon-type gradient inequality,
\[
\|\sE'(f)\|_{L^p(S^2)} \geq Z|\sE(f) - \sE(f_\infty)|^{1/2},
\]
has been obtained by Kwon \cite[Theorem 4.2]{KwonThesis} for maps $f:S^2\to N$, where $N$ is a closed Riemannian manifold and $f$ is close to a harmonic map $f_\infty$ in the sense that
\[
\|f - f_\infty\|_{W^{2,p}(S^2)} < \sigma,
\]
where $p$ is restricted to the range $1 < p \leq 2$, and $f_\infty$ is assumed to be \emph{integrable} in the sense of \cite[Definitions 4.3 or 4.4 and Proposition 4.1]{KwonThesis}. Her \cite[Proposition 4.1]{KwonThesis} quotes results of Simon \cite[pp. 270--272]{Simon_1985} and Adams and Simon \cite{Adams_Simon_1988}.

The result \cite[Lemma 3.3]{Liu_Yang_2010} due to Liu and Yang is another example of an optimal {\L}ojasiewicz--Simon-type gradient inequality for the harmonic map energy function, but restricted to the setting of maps $f:S^2\to N$, where $N$ is a K{\"a}hler manifold of complex dimension $n \geq 1$ and nonnegative bisectional curvature, and the energy $\sE(f)$ is sufficiently small. The result of Liu and Yang generalizes that of Topping \cite[Lemma 1]{Topping_1997}, who assumes that $N = S^2$.

For the Yamabe function, an optimal {\L}ojasiewicz--Simon gradient inequality has been obtained by Carlotto, Chodosh, and Rubinstein \cite{Carlotto_Chodosh_Rubinstein_2015} under the hypothesis that the critical point is \emph{integrable} in the sense of their
\cite[Definition 8]{Carlotto_Chodosh_Rubinstein_2015}, a condition that they observe in \cite[Lemma 9]{Carlotto_Chodosh_Rubinstein_2015} (quoting \cite[Lemma 1]{Adams_Simon_1988} due to Adams and Simon) is equivalent to a function on Euclidean space given by the \emph{Lyapunov--Schmidt reduction} of $\sE$ being constant on an open neighborhood of the critical point.

For the Yang--Mills energy function for connections on a principal $U(n)$-bundle over a closed Riemann surface, an optimal {\L}ojasiewicz--Simon gradient inequality, has been obtained by R\r{a}de \cite[Proposition 7.2]{Rade_1992} when the Yang-Mills connection has trivial stabilizer (equivalent to the center of $\U(n)$) in the group of gauge transformations.

Given the desirability of treating an energy function as a \emph{Morse function} whenever possible, for example in the spirit of Atiyah and Bott \cite{Atiyah_Bott_1983} for the Yang-Mills equation over Riemann surfaces, it is useful to rephrase these integrability conditions in the spirit of Morse theory.

\begin{defn}[Morse--Bott function]
\label{defn:Morse-Bott_function}
(See Austin and Braam \cite[Section 3.1]{Austin_Braam_1995}.)
Let $\sB$ be a smooth Banach manifold, let $\sE: \sB \to \RR$ be a $C^2$ function, and define
\[
  \Crit\sE := \{x\in\sB:\sE'(x) = 0\}.
\]
A smooth submanifold $\sC \hookrightarrow \sB$ is called a \emph{nondegenerate critical submanifold of $\sB$} if $\sC \subset \Crit\sE$ and
\begin{equation}
\label{eq:Nondegenerate_critical_submanifold}
(T\sC)_x = \Ker \sE''(x), \quad\forall\,x\in \sC,
\end{equation}
where $\sE''(x):(T\sB)_x \to (T\sB)_x^*$ is the Hessian of $\sE$ at the point $x \in \sC$. One calls $\sE$ a \emph{Morse--Bott function} if its critical set $\Crit\sE$ consists of nondegenerate critical submanifolds.

We say that a $C^2$ function $\sE:\sB\to\RR$ is \emph{Morse--Bott at a point $x_\infty \in \Crit\sE$} if there is an open neighborhood $\sU\subset\sB$ of $x_\infty$ such that $\sU\cap\Crit\sE$ is a relatively open, smooth submanifold of $\sB$ and \eqref{eq:Nondegenerate_critical_submanifold} holds at $x_\infty$.
\end{defn}

Definition \ref{defn:Morse-Bott_function} is a restatement of definitions of a Morse--Bott function on a finite-dimensional manifold, but we omit the condition that $\sC$ be compact and connected as in Nicolaescu \cite[Definition 2.41]{Nicolaescu_morse_theory} or the condition that $\sC$ be compact in Bott \cite[Definition, p. 248]{Bott_1954}. Note that if $\sB$ is a Riemannian manifold and $\sN$ is the normal bundle of $\sC \hookrightarrow \sB$, so $\sN_x = (T\sC)_x^\perp$ for all $x \in \sC$, where $(T\sC)_x^\perp$ is the orthogonal complement of $(T\sC)_x$ in $(T\sB)_x$, then \eqref{eq:Nondegenerate_critical_submanifold} is equivalent to the assertion that the restriction of the Hessian to the fibers of the normal bundle of $\sC$,
\[
\sE''(x):\sN_x \to (T\sB)_x^*,
\]
is \emph{injective} for all $x \in \sC$; using the Riemannian metric on $\sB$ to identify $(T\sB)_x^* \cong (T\sB)_x$, we see that $\sE''(x):\sN_x \cong \sN_x$ is an isomorphism for all $x \in \sC$. In other words, the condition \eqref{eq:Nondegenerate_critical_submanifold} is equivalent to the assertion that the Hessian of $\sE$ is an automorphism of the normal bundle $\sN$ when $\sB$ has a Riemannian metric.

The Yang-Mills energy function for connections with trivial stabilizer on a principal $G$-bundle over $X$ is Morse--Bott when $X$ is a closed Riemann surface --- see the article by Atiyah and Bott \cite{Atiyah_Bott_1983} and the discussion by Swoboda \cite[p. 161]{Swoboda_2012}. However, it appears difficult to extend this result to the case where $X$ is a closed, four-dimensional Riemannian manifold. To gain a sense of the difficulty, see the analysis by Bourguignon and Lawson \cite{Bourguignon_Lawson_1981} and Taubes \cite{TauStab} of the Hessian for the Yang-Mills energy function when $X = S^4$ with its standard round metric of radius one. For a development of Morse--Bott theory and a discussion of and references to its numerous applications, we refer to Austin and Braam \cite{Austin_Braam_1995}. Given a Morse--Bott function, we have the

\begin{mainthm}[Optimal {\L}ojasiewicz--Simon gradient inequality for Morse--Bott functions on Banach spaces]
\label{mainthm:Optimal_Lojasiewicz-Simon_gradient_inequality_Morse-Bott_energy_functional}  
Assume the hypotheses of Theorem \ref{mainthm:Lojasiewicz-Simon_gradient_inequality} or of Theorem \ref{mainthm:Lojasiewicz-Simon_gradient_inequality2}. 
%TODO Do we need index zero? Analytic?
If $\sM$ is $C^1$ and $\sE$ is a Morse--Bott function at $x_\infty$ in the sense of Definition \ref{defn:Morse-Bott_function}, then the conclusions of Theorem \ref{mainthm:Lojasiewicz-Simon_gradient_inequality} or \ref{mainthm:Lojasiewicz-Simon_gradient_inequality2} hold with $\theta = 1/2$.
\end{mainthm}

We refer to Feehan \cite[Appendix C]{Feehan_lojasiewicz_inequality_all_dimensions_morse-bott} for a discussion of integrability and the Morse--Bott condition for the harmonic map energy function, together with examples.

\begin{rmk}[Previous versions of the optimal {\L}ojasiewicz--Simon gradient inequality]
\label{rmk:Optimal_Lojasiewicz-Simon_gradient_inequality_Morse-Bott_energy_functional}
Special cases of Theorem \ref{mainthm:Optimal_Lojasiewicz-Simon_gradient_inequality_Morse-Bott_energy_functional}, were proved earlier by Chill \cite[Corollary 3.12]{Chill_2003}, \cite[Corollary 4]{Chill_2006}, Haraux and Jendoubi \cite[Theorem 2.1]{Haraux_Jendoubi_2007}, and Simon \cite[Lemma 3.13.1]{Simon_1996}.
\end{rmk}

\subsection{{\L}ojasiewicz--Simon gradient inequality for the harmonic map energy function}
\label{subsec:Lojasiewicz-Simon_gradient_inequality_harmonic_map_functional}
We describe a consequence of Theorem \ref{mainthm:Lojasiewicz-Simon_gradient_inequality} for the harmonic map energy function. For background on harmonic maps, we refer to H{\'e}lein \cite{Helein_harmonic_maps}, Jost \cite{Jost_riemannian_geometry_geometric_analysis}, Simon \cite{Simon_1996}, Struwe \cite{Struwe_variational_methods}, and references cited therein. We begin with the

\begin{defn}[Harmonic map energy function]
\label{defn:Harmonic_map_energy_functional}
Let $(M,g)$ and $(N,h)$ be a pair of closed, smooth Riemannian manifolds. One defines the \emph{harmonic map energy function} by
\begin{equation}
\label{eq:Harmonic_map_energy_functional}
\sE_{g,h}(f)
:=
\frac{1}{2} \int_M |df|_{g,h}^2 \,d\vol_g,
\end{equation}
for smooth maps $f:M\to N$, where $df:TM \to TN$ is the differential map.
\end{defn}

When clear from the context, we omit explicit mention of the Riemannian metrics $g$ on $M$ and $h$ on $N$ and write $\sE = \sE_{g,h}$. Although initially defined for smooth maps, the energy function $\sE$ in Definition \ref{defn:Harmonic_map_energy_functional}, extends to the case of Sobolev maps of class $W^{1,2}$. To define the gradient $\sM = \sM_{g,h}$ of the energy function $\sE$ in equation \eqref{eq:Harmonic_map_energy_functional} with respect to the $L^2$ metric on $C^\infty(M;N)$, we first choose an isometric embedding, $(N,h) \hookrightarrow \RR^n$ for a sufficiently large $n$ (courtesy of the isometric embedding theorem due to Nash \cite{Nash_1956}), and recall that by \cite[Lemma 1.2.4]{Helein_harmonic_maps} we have
\begin{equation}
\label{eq:Gradient_harmonic_map_operator}
\sM(f) = \Delta_g f - A_h(f)(df,df),
\end{equation}
as in \cite[Equations (2.2)(iii) and (iv)]{Simon_1996}. Here, $A_h$ denotes the second fundamental form of the isometric embedding $(N,h) \subset \RR^n$, and
\begin{equation}
\label{eq:Laplace-Beltrami_operator}
\Delta_g
:=
-\divg_g \grad_g
=
d^{*,g}d
=
-\frac{1}{\sqrt{\det g}} \frac{\partial}{\partial x^\beta}
\left(\sqrt{\det g}\, \frac{\partial f}{\partial x^\alpha} \right)
\end{equation}
denotes the Laplace-Beltrami operator for $(M,g)$ (with the opposite sign convention to that of \cite[Equations (1.14) and (1.33)]{Chavel}) acting on the scalar components $f^i$ of $f = (f^1,\ldots,f^n)$, and $\{x^\alpha\}$ denote local coordinates on $M$.

Given a smooth map $f:M\to N$, an isometric embedding $(N,h) \subset \RR^n$, a non-negative integer $k$, and constant $p \in [1,\infty)$, we define the Sobolev norms
\[
\|f\|_{W^{k,p}(M)} := \left(\sum_{i=1}^n \|f^i\|_{W^{k,p}(M)}^p\right)^{1/p},
\]
with
\[
\|f^i\|_{W^{k,p}(M)} := \left(\sum_{j=0}^k \int_M |(\nabla^g)^j f^i|^p \,d\vol_g\right)^{1/p},
\]
where $\nabla^g$ denotes the Levi-Civita connection on $TM$ and all associated bundles (that is, $T^*M$ and their tensor products). If $k=0$, then we denote $\|f\|_{W^{0,p}(M)} = \|f\|_{L^p(M)}$. For $p \in [1,\infty)$ and nonnegative integers $k$, we use \cite[Theorem 3.12]{AdamsFournier} (applied to $W^{k,p}(M;\RR^n)$ and noting that $M$ is a closed manifold) and Banach-space duality to define
\[
W^{-k,p'}(M;\RR^n) := \left(W^{k,p}(M;\RR^n)\right)^*,
\]
where $p'\in (1,\infty)$ is the dual exponent defined by $1/p+1/p'=1$. Elements of the Banach space dual $(W^{k,p}(M;\RR^n))^*$ may be characterized via \cite[Section 3.10]{AdamsFournier} as distributions in the Schwartz space $\sD'(M;\RR^n)$ \cite[Section 1.57]{AdamsFournier}.

We note that if $(N,h)$ is real analytic, then the isometric embedding $(N,h) \subset \RR^n$ may also be chosen to be analytic by the analytic isometric embedding theorem due to Nash \cite{Nash_1966}, with a simplified proof due to Greene and Jacobowitz \cite{Greene_Jacobowitz_1971}).

One says that a map $f \in W^{1,2}(M;N)$ is \emph{weakly harmonic} \cite[Definition 1.4.9]{Helein_harmonic_maps} if it is a critical point of the energy function \eqref{eq:Harmonic_map_energy_functional}, that is,
\[
\sE'(f) = 0.
\]
A well-known result due to H\'elein \cite[Theorem 4.1.1]{Helein_harmonic_maps} tells us that if $M$ has dimension $d=2$, then $f \in C^\infty(M;N)$; for $d \geq 3$, regularity results are far more limited --- see, for example, \cite[Theorem 4.3.1]{Helein_harmonic_maps} due to Bethuel.

The statement of the forthcoming Theorem \ref{mainthm:Lojasiewicz-Simon_gradient_inequality_energy_functional_Riemannian_manifolds} includes the most delicate dimension for the source Riemannian manifold $(M,g)$, namely the case where $M$ has dimension $d=2$. Following the landmark articles by Sacks and Uhlenbeck \cite{Sacks_Uhlenbeck_1981, Sacks_Uhlenbeck_1982}, the case where the domain manifold $M$ has dimension two is well-known to be critical.

%COMMENT-PF-10-21-2015 Does $M$ need to be oriented?
\begin{mainthm}[{\L}ojasiewicz--Simon $W^{k-2,p}$ gradient inequality for the energy function for maps between pairs of Riemannian manifolds]
\label{mainthm:Lojasiewicz-Simon_gradient_inequality_energy_functional_Riemannian_manifolds}
(See Feehan and Maridakis \cite[Theorem 5]{Feehan_Maridakis_Lojasiewicz-Simon_application_harmonic_maps}.)  
Let $d\geq 2$ and $k \geq 1$ be integers and $p\in (1,\infty)$ be such that $kp > d$. Let $(M,g)$ and $(N,h)$ be closed, smooth Riemannian manifolds, with $M$ of dimension $d$. If $(N,h)$ is real analytic (respectively, $C^\infty$) and $f\in W^{k,p}(M;N)$, then the gradient map for the energy function $\sE:W^{k,p}(M; N)\to\RR$ in \eqref{eq:Harmonic_map_energy_functional},
\[
W^{k,p}(M; N) \ni f \mapsto \sM(f) \in W^{k-2,p}(M; f^*TN) \subset W^{k-2,p}(M; \RR^n),
\]
is a real analytic (respectively, $C^\infty$) map of Banach spaces. If $(N,h)$ is real analytic and $f_\infty \in W^{k,p}(M; N)$ is a weakly harmonic map, then there are constants $Z \in (0, \infty)$ and $\sigma \in (0,1]$ and $\theta \in [1/2,1)$, depending on $f_\infty$, $g$, $h$, $k$, $p$, with the following significance. If $f\in W^{k,p}(M;N)$ obeys the $W^{k,p}$ \emph{{\L}ojasiewicz--Simon neighborhood} condition,
\begin{equation}
\label{eq:Lojasiewicz-Simon_gradient_inequality_harmonic_map_neighborhood_Riemannian_manifold}
\|f - f_\infty\|_{W^{k,p}(M)} < \sigma,
\end{equation}
then the harmonic map energy function
\eqref{eq:Harmonic_map_energy_functional} obeys the
\emph{{\L}ojasiewicz--Simon gradient inequality},
\begin{equation}
\label{eq:Lojasiewicz-Simon_gradient_inequality_harmonic_map_energy_functional_Riemannian_manifold}
\|\sM(f)\|_{W^{k-2,p}(M;f^*TN)}
\geq
Z|\sE(f) - \sE(f_\infty)|^\theta.
\end{equation}
Furthermore, if the hypothesis that $(N,h)$ is analytic is replaced by the condition that $\sE$ is Morse--Bott at $f_\infty$, then \eqref{eq:Lojasiewicz-Simon_gradient_inequality_harmonic_map_energy_functional_Riemannian_manifold} holds with the optimal exponent $\theta=1/2$.
\end{mainthm}

\begin{rmk}[On the hypotheses of Theorem \ref{mainthm:Lojasiewicz-Simon_gradient_inequality_energy_functional_Riemannian_manifolds}]
  When $k=d$ and $p=1$, then $W^{d,1}(M;\RR) \subset C(M;\RR)$ is a continuous embedding by \cite[Theorem 4.12]{AdamsFournier} and $W^{d,1}(M;\RR)$ is a Banach algebra by \cite[Theorem 4.39]{AdamsFournier}. In particular, $W^{d,1}(M; N)$ is a real analytic Banach manifold and the harmonic map energy function $\sE: W^{d,1}(M; N) \to \RR$ is real analytic; see Feehan and Maridakis \cite{Feehan_Maridakis_Lojasiewicz-Simon_application_harmonic_maps} for details. However,
  \[
    \sM'(f_\infty):W^{d,1}(M; f_\infty^*TN) \to W^{d-2,1}(M; f_\infty^*TN)
  \]
need not be a Fredholm operator. Indeed, when $d=2$, failure of the Fredholm property for
\[
\sM'(f_\infty):W^{2,1}(M; f_\infty^*TN) \to L^1(M; f_\infty^*TN)
\]
(unless $L^1(M; f_\infty^*TN)$ is replaced, for example, by a Hardy $H^1$ space) can be inferred from calculations described by H\'elein \cite{Helein_harmonic_maps}.
\end{rmk}

\begin{rmk}[Previous versions of the {\L}ojasiewicz--Simon gradient inequality for the harmonic map energy function]
Topping \cite[Lemma 1]{Topping_1997} proved a {\L}ojasiewicz-type gradient inequality for maps $f:S^2 \to S^2$ with small energy, with the latter criterion replacing the usual small $C^{2,\alpha}(M;\RR^n)$-norm criterion of Simon for the difference between a map and a critical point \cite[Theorem 3]{Simon_1983}. Simon uses a $C^2(M;\RR^n)$ norm to measure distance between maps $f:M \to N$ in \cite[Equation (4.27)]{Simon_1985}. Topping's result is generalized by Liu and Yang in \cite[Lemma 3.3]{Liu_Yang_2010}. Kwon \cite[Theorem 4.2]{KwonThesis} obtains a {\L}ojasiewicz-type gradient inequality for maps $f:S^2 \to N$ that are $W^{2,p}(S^2;\RR^n)$-close to a harmonic map, with $1 < p \leq 2$.
\end{rmk}

Theorem \ref{mainthm:Lojasiewicz-Simon_gradient_inequality2} leads in turn to the following refinement of Theorem \ref{mainthm:Lojasiewicz-Simon_gradient_inequality_energy_functional_Riemannian_manifolds}.

\begin{maincor}[{\L}ojasiewicz--Simon $L^2$ gradient inequality for the energy function for maps between pairs of Riemannian manifolds]
\label{maincor:Lojasiewicz-Simon_gradient_inequality_energy_functional_Riemannian_manifolds_L2}
(See Feehan and Maridakis \cite[Corollary 6]{Feehan_Maridakis_Lojasiewicz-Simon_application_harmonic_maps}.)    
Assume the hypotheses of Theorem \ref{mainthm:Lojasiewicz-Simon_gradient_inequality_energy_functional_Riemannian_manifolds} and, in addition, require that $k$ and $p$ obey
\begin{enumerate}
\item \label{item:maincor_LS_grad_ineq_L2_d=2_k=1}
$d=2$ and $k=1$ and $2 < p < \infty$; or

\item \label{item:maincor_LS_grad_ineq_L2_d=3_k=1}
$d=3$ and $k=1$ and $3 < p \leq 6$; or

\item \label{item:maincor_LS_grad_ineq_L2_dgeq2_kgeq2}
$d\geq 2$ and $k\geq 2$ and $2 \leq p < \infty$ with $kp > d$.
%COMMENT: Can we not have $1 < p < \infty$ with $kp > d$?
\end{enumerate}
If $f\in W^{k,p}(M;N)$ obeys the $W^{k,p}$ \emph{{\L}ojasiewicz--Simon neighborhood} condition \eqref{eq:Lojasiewicz-Simon_gradient_inequality_harmonic_map_neighborhood_Riemannian_manifold}, then the harmonic map energy function \eqref{eq:Harmonic_map_energy_functional} obeys the \emph{{\L}ojasiewicz--Simon $L^2$ gradient inequality},
\begin{equation}
\label{eq:Lojasiewicz-Simon_gradient_inequality_harmonic_map_energy_functional_Riemannian_manifold_L2}
\|\sM(f)\|_{L^2(M;f^*TN)}
\geq
Z|\sE(f) - \sE(f_\infty)|^\theta.
\end{equation}
Furthermore, if the hypothesis that $(N,h)$ be analytic is replaced by the condition that $\sE$ be Morse--Bott at $f_\infty$, then \eqref{eq:Lojasiewicz-Simon_gradient_inequality_harmonic_map_energy_functional_Riemannian_manifold_L2} holds with the optimal exponent $\theta=1/2$.
\end{maincor}

The proofs of Theorem \ref{mainthm:Lojasiewicz-Simon_gradient_inequality_energy_functional_Riemannian_manifolds} and Corollary \ref{maincor:Lojasiewicz-Simon_gradient_inequality_energy_functional_Riemannian_manifolds_L2} are provided in Feehan and Maridakis \cite{Feehan_Maridakis_Lojasiewicz-Simon_application_harmonic_maps}.

\begin{rmk}[Application to proof of Simon's $L^2$ gradient inequality for the energy function for maps between pairs of Riemannian manifolds]
Simon's statement \cite[Theorem 3]{Simon_1983}, \cite[Equation (4.27)]{Simon_1985} of the $L^2$ gradient inequality for the energy function for maps from a closed Riemannian manifold into a closed, real analytic Riemannian manifold is identical to that of Corollary \ref{maincor:Lojasiewicz-Simon_gradient_inequality_energy_functional_Riemannian_manifolds_L2}, except that it applies to $C^{2,\alpha}$ (rather than $W^{k,p}$) maps (for $\alpha \in (0,1)$) and the condition \eqref{eq:Lojasiewicz-Simon_gradient_inequality_harmonic_map_neighborhood_Riemannian_manifold} is replaced by
\[
\|f - f_\infty\|_{C^{2,\alpha}(M;\RR^n)} < \sigma,
\]
Simon's result \cite[Theorem 3]{Simon_1983}, \cite[Equation (4.27)]{Simon_1985} follows immediately from Corollary \ref{maincor:Lojasiewicz-Simon_gradient_inequality_energy_functional_Riemannian_manifolds_L2} and the Sobolev Embedding \cite[Theorem 4.12]{AdamsFournier} by choosing $k\geq 1$ and $p \in (1,\infty)$ with $kp>d$ so that there is a continuous Sobolev embedding, $C^{2,\alpha}(M;\RR) \subset W^{k,p}(M;\RR)$ and thus
\[
\|f - f_\infty\|_{W^{k,p}(M;\RR^n)} \leq C\|f - f_\infty\|_{C^{2,\alpha}(M;\RR^n)},
\]
for some constant, $C = C(g,h,k,p,\alpha) \in [1,\infty)$.
\end{rmk}

\begin{rmk}[Exclusion of the case $d\geq 4$ and $k=1$ in Corollary \ref{maincor:Lojasiewicz-Simon_gradient_inequality_energy_functional_Riemannian_manifolds_L2}]
\label{rmk:Lojasiewicz-Simon_gradient_inequality_energy_functional_manifolds_L2_exclude_d_geq_4}
The proofs of Items \eqref{item:maincor_LS_grad_ineq_L2_d=2_k=1} and \eqref{item:maincor_LS_grad_ineq_L2_d=3_k=1} require that $p$ obey
\[
  (p')^* = \frac{dp}{d(p-1)-p} \geq 2,
\]
namely $dp \geq 2d(p-1) - 2p = 2dp-2d-2p$, or equivalently, $dp \leq 2d+2p$, that is,
\[
  p \leq \frac{2d}{d-2}.
\]
But the condition $kp>d$ for $k=1$ implies $p>d$ and so $d$ must obey $d < 2d/(d-2)$, that is, $d-2 < 2$ or $d<4$.
\end{rmk}

\begin{rmk}[Relaxing the condition $p \geq 2$ in Item \eqref{item:maincor_LS_grad_ineq_L2_dgeq2_kgeq2} of Corollary \ref{maincor:Lojasiewicz-Simon_gradient_inequality_energy_functional_Riemannian_manifolds_L2}]
When $k \geq 3$, the condition $p\geq 2$ in Item \eqref{item:maincor_LS_grad_ineq_L2_dgeq2_kgeq2} of Corollary \ref{maincor:Lojasiewicz-Simon_gradient_inequality_energy_functional_Riemannian_manifolds_L2} can be relaxed using the Sobolev embedding \cite[Theorem 4.12]{AdamsFournier}.
\end{rmk}

\subsection{{\L}ojasiewicz--Simon gradient inequalities for boson and fermion coupled Yang--Mills energy functions}
\label{subsec:Lojasiewicz-Simon_gradient_inequality_boson_fermion_Yang--Mills_function}
Coupled versions of the pure Yang--Mills energy function play an essential role in Theoretical Physics and coupled Yang--Mills-type equations lie at the heart of numerous areas of research in Geometric Analysis and Mathematical Physics. We begin with a definition (due to Parker \cite{ParkerGauge}) of two coupled Yang--Mills energy functions. We refer the reader to Feehan and Maridakis \cite{Feehan_Maridakis_Lojasiewicz-Simon_coupled_Yang-Mills_v6} for further details and discussion.

\begin{defn}[Boson and fermion coupled Yang--Mills energy functions]
\label{defn:Boson_and_fermion_coupled_Yang--Mills_energy_function}
(See Parker \cite[Section 2]{ParkerGauge}.)
Let $(X,g)$ be a closed, smooth Riemannian manifold of dimension $d \geq 2$, let $G$ be a compact Lie group, let $P$ be a smooth principal $G$-bundle over $X$, let $\EE$ be a complex, finite-dimensional $G$-module equipped with a $G$-invariant Hermitian inner product, let $\varrho: G \to \Aut_\CC(\EE)$ be a unitary representation \cite[Definitions 2.1.1 and 2.16]{BrockertomDieck}, let $E = P\times_\varrho\EE$ be a smooth Hermitian vector bundle over $X$, and let $m$ and $s$ be smooth real-valued functions on $X$.

We define the \emph{boson coupled Yang--Mills energy function} by
\begin{equation}
\label{eq:Boson_Yang--Mills_energy_function}
\sE_g(A,\Phi)
:=
\frac{1}{2} \int_X \left(|F_A|^2 + |\nabla_A \Phi|^2 -  m|\Phi|^2 - s|\Phi|^4\right)
\,d\vol_g,
\end{equation}
for all smooth connections $A$ on $P$ and smooth sections $\Phi$ of $E$, where
\[
\nabla_A: C^\infty(X;E) \to C^\infty(T^*X\otimes E)
\]
is the covariant derivative induced on $E$ by the connection $A$ on $P$, and $F_A \in \Omega^2(X;\ad P)$ is the curvature of $A$, and $\ad P := P\times_{\Ad}\fg$ denotes the real vector bundle associated to $P$ by the adjoint representation of $G$ on its Lie algebra,
$\Ad:G \ni u \to \Ad_u \in \Aut(\fg)$, with fiber metric defined through an $\Ad(G)$-invariant inner product on $\fg$.

Suppose that $X$ admits a \spinc structure comprising a Hermitian vector bundle $W$ over $X$ and a \emph{Clifford multiplication map} $\rho:T^*X \to \End_\CC(W)$. Thus,
\begin{equation}
\label{eq:Clifford_multiplication}
\rho(\alpha)^2 = -g(\alpha,\alpha)\,\id_W, \quad\forall\, \alpha \in \Omega^1(X),
\end{equation}
and
\[
D_A := \rho\circ\nabla_A: C^\infty(X;W\otimes E) \to C^\infty(X;W\otimes E),
\]
is the corresponding \emph{Dirac operator} \cite[Appendix D]{LM}, \cite[Sections 1.1 and 1.2]{KMBook}, where $\nabla_A$ denotes the covariant derivative induced on $\otimes^n(T^*X)\otimes E$ (for $n \geq 0$) and $W\otimes E$ by the connection $A$ on $P$ and the Levi-Civita connection for the metric $g$ on $TX$.

We define the \emph{fermion coupled Yang--Mills energy function} by
\begin{equation}
\label{eq:Fermion_Yang--Mills_energy_function}
\sF_g(A,\Psi)
:=
\frac{1}{2} \int_X \left(|F_A|^2 + \langle \Psi, D_A\Psi\rangle -  m|\Psi|^2\right)
\,d\vol_g,
\end{equation}
for all smooth connections $A$ on $P$ and smooth sections $\Psi$ of $W\otimes E$.
\end{defn}

We recall from \cite[Corollary D.4]{LM} that a closed orientable smooth manifold $X$ admits a \spinc structure if and only if the second Stiefel-Whitney class $w_2(X) \in H^2(X;\ZZ/2\ZZ)$ is the mod $2$ reduction of an integral class. One calls $W$ the \emph{fundamental spinor bundle} and it carries irreducible representations of $\Spinc(d)$; when $X$ is even-dimensional, there is a splitting $W = W^+\oplus W^-$ and Clifford multiplication restricts to give $\rho: T^*X \to \Hom_\CC(W^\pm, W^\mp)$ \cite[Definition D.9]{LM}.

Although initially defined for smooth connections and sections, the energy functions $\sE_g$ and $\sF_g$ in Definition \ref{defn:Boson_and_fermion_coupled_Yang--Mills_energy_function}, extend to the case of Sobolev connections and sections of class $W^{1,2}$.

A short calculation shows that the gradient of the boson coupled Yang--Mills energy function $\sE_g$ in \eqref{eq:Boson_Yang--Mills_energy_function} with respect to the $L^2$ metric on $C^\infty(X;\Lambda^1\otimes\ad P\oplus E)$,
\begin{equation}
\label{eq:Definition_gradient_boson_coupled_Yang--Mills_energy_function}
\left(\sM_g(A,\Phi), (a,\phi)\right)_{L^2(X,g)}
:=
\left.\frac{d}{dt}\sE_g(A+ta, \Phi+t\phi)\right|_{t=0}
=
\sE_g'(A,\Phi)(a,\phi),
\end{equation}
for all $(a,\phi) \in C^\infty(X;\Lambda^1\otimes\ad P\oplus E)$, is given by
\begin{multline}
\label{eq:Gradient_boson_coupled_Yang--Mills_energy_function}
\left(\sM_g(A,\Phi), (a,\phi)\right)_{L^2(X,g)}
\\
=
(d_A^*F_A, a)_{L^2(X)} + \Real (\nabla_A^*\nabla_A \Phi, \phi )_{L^2(X)} + \Real(\nabla_A\Phi, \rho(a)\Phi )_{L^2(X)}
\\
- \Real ( m\Phi, \phi )_{L^2(X)}
- 2\Real \int_X s|\Phi|^2 \langle \Phi,\phi\rangle\, d\vol_g,
\end{multline}
where $d_A^* = d_A^{*,g}: \Omega^l(X; \ad P) \to \Omega^{l-1}(X; \ad P)$ is the $L^2$ adjoint of the exterior covariant derivative $d_A:\Omega^l(X; \ad P) \to \Omega^{l+1}(X; \ad P)$, for integers $l\geq 0$. As customary, we let
\[
\Lambda^l = \wedge^l(T^*X)
\]
denote the vector bundle over $X$ whose fiber $\wedge^l(T_x^*X)$ over each point $x \in X$ is the $l$-th exterior power of the cotangent space $T_x^*X$, with $\wedge^0(T^*X) := X\times\RR$ and $\wedge^1(T^*X) = T^*X$. For any integer $l\geq 0$, we abbreviate $\Omega^l(X;\ad P) = C^\infty(\Lambda^l\otimes\ad P)$.

We call $(A,\Phi)$ a \emph{boson Yang--Mills pair} (with respect to the Riemannian metric $g$ on $X$) if it is a critical point for $\sE_g$, that is, $\sM_g(A,\Phi) = 0$.

Similarly, one finds that the gradient of the fermion coupled Yang--Mills energy function $\sF_g$ in \eqref{eq:Fermion_Yang--Mills_energy_function} with respect to the $L^2$ metric on $C^\infty(X;\Lambda^1\otimes\ad P\oplus W\otimes E)$,
\begin{equation}
\label{eq:Definition_gradient_fermion_coupled_Yang--Mills_energy_function}
\left(\sM_g(A,\Psi), (a,\psi)\right)_{L^2(X,g)}
:=
\left.\frac{d}{dt}\sF_g(A+ta, \Psi+t\psi)\right|_{t=0}
=
\sF_g'(A,\Psi)(a,\psi),
\end{equation}
for all $(a,\psi) \in C^\infty(X;\Lambda^1\otimes\ad P\oplus W\otimes E)$, is given by
\begin{multline}
\label{eq:Gradient_fermion_coupled_Yang--Mills_energy_function}
\left(\sM_g(A,\Psi), (a,\psi)\right)_{L^2(X,g)}
=
(d_A^*F_A, a)_{L^2(X)} + \Real(D_A \Psi - m\Psi, \psi )_{L^2(X)}
\\
+ \frac{1}{2}(\Psi, \rho(a)\Psi )_{L^2(X)},
\end{multline}
where the action of $a \in \Omega^1(X;\ad P)$ on $\Psi \in C^\infty(X;W\otimes E)$ is defined for all $\alpha \in \Omega^1(X)$ and $\xi \in C^\infty(X;\ad P)$ and $\phi \in C^\infty(X;W)$ and $\eta \in C^\infty(X;E)$ by
\[
\rho(\alpha\otimes\xi)(\phi\otimes \eta)
:=
\rho(\alpha)\phi\otimes \varrho_*(\xi)\eta,
\]
where $\varrho_*:\fg \to \End_\CC(\EE)$ is the representation of the Lie algebra induced by the representation $\varrho: G \to \End_\CC(\EE)$ of the Lie group.

We call $(A,\Psi)$ a \emph{fermion Yang--Mills pair} (with respect to the Riemannian metric $g$ on $X$) if it is a critical point for $\sF_g$, that is, $\sM_g(A,\Psi) = 0$.

Note that both the boson and fermion coupled Yang--Mills energy functions reduce to the pure \emph{Yang--Mills energy function} when $\Phi \equiv 0$ or $\Psi \equiv 0$, respectively,
\begin{equation}
\label{eq:Yang--Mills_energy_function}
\sE_g(A)  := \frac{1}{2}\int_X |F_A|^2\,d\vol_g,
\end{equation}
and $A$ is a \emph{Yang--Mills connection} (with respect to the Riemannian metric $g$ on $X$) if it is a critical point for $\sE_g$, that is,
\[
\sM_g(A) = d_A^{*,g}F_A = 0.
\]
Given a Hermitian or Riemannian vector bundle $V$ over $X$ and covariant derivative $\nabla_A$ which is compatible with the fiber metric on $V$, we denote the Banach space of sections of $V$ of Sobolev class $W^{k,p}$, for any $k\in \NN$ and $p \in [1,\infty]$, by $W_A^{k,p}(X; V)$, with norm,
\begin{equation}
\label{eq:Sobolev_norm_WAkp_sections_vector_bundle_over_manifold_finite_p}
\|v\|_{W_A^{k,p}(X)} := \left(\sum_{j=0}^k \int_X |\nabla_A^j v|^p\,d\vol_g \right)^{1/p},
\end{equation}
when $1\leq p<\infty$ and
\begin{equation}
\label{eq:Sobolev_norm_WAkp_sections_vector_bundle_over_manifold_infinite_p}
\|v\|_{W_A^{k,\infty}(X)} := \sum_{j=0}^k \esssup_X |\nabla_A^j v|,
\end{equation}
when $p=\infty$, where $v \in W_A^{k,p}(X; V)$. If $k=0$, then we denote
\[
  \|v\|_{W^{0,p}(X)} = \|v\|_{L^p(X)}.
\]
For $p \in [1,\infty)$ and nonnegative integers $k$, we use \cite[Theorem 3.12]{AdamsFournier} (applied to $W_A^{k,p}(X;V)$ and noting that $X$ is a closed manifold) and Banach-space duality to define
\[
W_A^{-k,p'}(X;V) := \left(W_A^{k,p}(X;V)\right)^*,
\]
where $p'\in (1,\infty]$ is the dual exponent defined by $1/p+1/p'=1$ and we use the fiber metric on $V$ to replace $V^*$ by $V$ on the left-hand side. Elements of the Banach-space dual $(W_A^{k,p}(X;V))^*$ may be characterized via \cite[Section 3.10]{AdamsFournier} as distributions in the Schwartz space $\sD'(X;V)$ \cite[Section 1.57]{AdamsFournier}.

As our first application of Theorem \ref{mainthm:Lojasiewicz-Simon_gradient_inequality}, we have the following generalization of \cite[Theorem 23.17]{Feehan_yang_mills_gradient_flow_v4} from the case of the pure Yang--Mills energy function \eqref{eq:Yang--Mills_energy_function}, when $p=2$ and $X$ has dimension $d=2,3$, or $4$, and R\r{a}de's \cite[Proposition 7.2]{Rade_1992}, when $p=2$ and $X$ has dimension $d=2$ or $3$.

\begin{mainthm}[{\L}ojasiewicz--Simon gradient inequality for the boson coupled Yang--Mills energy function]
\label{mainthm:Lojasiewicz-Simon_gradient_inequality_boson_Yang--Mills_energy_function}
(See Feehan and Maridakis \cite[Theorem 4 and Corollary 7]{Feehan_Maridakis_Lojasiewicz-Simon_coupled_Yang-Mills_v6}.)  
Let $(X,g)$ be a closed, smooth Riemannian manifold of dimension $d\geq 2$, let $G$ be a compact Lie group, let $P$ be a smooth principal $G$-bundle over $X$, and let $E = P\times_\varrho\EE$ be a smooth Hermitian vector bundle over $X$ defined by a finite-dimensional unitary representation $\varrho: G \to \Aut_\CC(\EE)$. Let $A_1$ be a $C^\infty$ reference connection on $P$ and let $(A_\infty,\Phi_\infty)$ be a boson coupled Yang--Mills pair on $(P,E)$ for $g$ of class $W^{1,q}$, with $q \in [2,\infty)$ obeying $q > d/2$. If $p \in [2,\infty)$ obeys $d/2 \leq p \leq q$, then the gradient map,
\[
\sM_g: (A_1,0)+W_{A_1}^{1,p}(X;\Lambda^1\otimes\ad P\oplus E)
\to W_{A_1}^{-1,p}(X;\Lambda^1\otimes\ad P\oplus E),
\]
is \emph{real analytic} and there are constants $Z \in (0, \infty)$ and $\sigma \in (0,1]$ and $\theta \in [1/2,1)$, depending on $A_1$, $(A_\infty,\Phi_\infty)$, $g$, $G$, $p$, and $q$, with the following significance. If $(A,\Phi)$ is a $W^{1,q}$ Sobolev pair on $(P,E)$ obeying the \emph{{\L}ojasiewicz--Simon neighborhood} condition,
%COMMENT-PF-11-20-2017 It would suffice for this condition to hold for $u(A,\Phi)$ and some $u \in \Aut^{2,q}(P)$
\begin{equation}
\label{eq:Lojasiewicz-Simon_gradient_inequality_boson_Yang--Mills_pair_neighborhood}
\|(A,\Phi) - (A_\infty,\Phi_\infty)\|_{W^{1,p}_{A_1}(X)} < \sigma,
\end{equation}
then the boson coupled Yang--Mills energy function \eqref{eq:Boson_Yang--Mills_energy_function} obeys the \emph{{\L}ojasiewicz--Simon gradient inequality}
\begin{equation}
\label{eq:Lojasiewicz-Simon_gradient_inequality_boson_Yang--Mills_energy_function}
\|\sM_g(A,\Phi)\|_{W^{-1,p}_{A_1}(X)}
\geq
Z|\sE_g(A,\Phi) - \sE_g(A_\infty,\Phi_\infty)|^\theta.
\end{equation}
Moreover, the inequality \eqref{eq:Lojasiewicz-Simon_gradient_inequality_boson_Yang--Mills_energy_function} holds with the Banach-space norm on $W_{A_1}^{-1,p}(X;\Lambda^1\otimes\ad P\oplus E)$ replaced by the Hilbert-space norm on $W_{A_1}^{-1,2}(X;\Lambda^1\otimes\ad P\oplus E)$.
\end{mainthm}

Similarly, for the fermion coupled Yang--Mills energy function, we have the

\begin{mainthm}[{\L}ojasiewicz--Simon gradient inequality for the fermion coupled Yang--Mills energy function]
\label{mainthm:Lojasiewicz-Simon_gradient_inequality_fermion_Yang--Mills_energy_function}
(See Feehan and Maridakis \cite[Theorem 6 and Corollary 8]{Feehan_Maridakis_Lojasiewicz-Simon_coupled_Yang-Mills_v6}.)    
Assume the hypotheses of Theorem
\ref{mainthm:Lojasiewicz-Simon_gradient_inequality_boson_Yang--Mills_energy_function},
except that we require that $X$ admit a \spinc structure $(\rho,W)$,
replace the role of $\sE_g$ in
\eqref{eq:Boson_Yang--Mills_energy_function} by $\sF_g$ in
\eqref{eq:Fermion_Yang--Mills_energy_function}, and replace the role
of the pair $(A,\Phi)$ and critical point $(A_\infty,\Phi_\infty)$ of
$\sE_g$ by the pair $(A,\Psi)$ and critical point
$(A_\infty,\Psi_\infty)$ of $\sF_g$, where $\Psi$ and $\Psi_\infty$
are sections of $W\otimes E$. Then the conclusions of Theorem \ref{mainthm:Lojasiewicz-Simon_gradient_inequality_boson_Yang--Mills_energy_function} hold \emph{mutatis mutandis} for $\sF_g$.
\end{mainthm}

The proofs of Theorems \ref{mainthm:Lojasiewicz-Simon_gradient_inequality_boson_Yang--Mills_energy_function} and \ref{mainthm:Lojasiewicz-Simon_gradient_inequality_fermion_Yang--Mills_energy_function} are provided in \cite{Feehan_Maridakis_Lojasiewicz-Simon_coupled_Yang-Mills_v6}.

\begin{rmk}[{\L}ojasiewicz--Simon gradient inequality for the Yang--Mills energy function over a Riemann surface]
\label{rmk:Lojasiewicz-Simon_gradient_inequality_Yang--Mills_Riemann_surface_Morse--Bott}
When $d=2$, it is known in many cases (see \cite{Feehan_lojasiewicz_inequality_ground_state}) that the pure Yang--Mills energy function obeys the Morse--Bott condition in the sense of Definition \ref{defn:Morse-Bott_function} and so by Theorem \ref{mainthm:Optimal_Lojasiewicz-Simon_gradient_inequality_Morse-Bott_energy_functional}, one has the optimal {\L}ojasiewicz--Simon exponent, $\theta = 1/2$, in those cases.
\end{rmk}

\subsection{Outline of the article}
\label{subsec:Outline}
In Section
\ref{sec:Lojasiewicz-Simon_gradient_inequality_analytic_and_Morse-Bott_energy_functionals},
we derive an abstract {\L}ojasiewicz--Simon gradient inequality for an
analytic function over a Banach space, proving Theorems
\ref{mainthm:Lojasiewicz-Simon_gradient_inequality_dualspace}, and \ref{mainthm:Lojasiewicz-Simon_gradient_inequality}, \ref{mainthm:Lojasiewicz-Simon_gradient_inequality2}, and for a Morse--Bott function over a Banach space, proving Theorem \ref{mainthm:Optimal_Lojasiewicz-Simon_gradient_inequality_Morse-Bott_energy_functional}.

We refer the reader to Feehan \cite[Appendix
C]{Feehan_lojasiewicz_inequality_all_dimensions_morse-bott} for a review
of the relationship between the Morse--Bott property and 
integrability in the setting of harmonic maps. In Feehan and Maridakis \cite[Appendix E]{Feehan_Maridakis_Lojasiewicz-Simon_coupled_Yang-Mills_v6}, we give a
review of Huang's \cite[Theorem 2.4.2 (i)]{Huang_2006} for the
{\L}ojasiewcz--Simon gradient inequality for analytic functions on
Banach spaces. Next, Feehan and Maridakis \cite[Appendix
D]{Feehan_Maridakis_Lojasiewicz-Simon_harmonic_maps_v5} provide a few
elementary observations from linear functional analysis that
illuminate the hypotheses of Theorems
\ref{mainthm:Lojasiewicz-Simon_gradient_inequality_dualspace} and
\ref{mainthm:Lojasiewicz-Simon_gradient_inequality}. Lastly,
Feehan and Maridakis \cite[Appendix
D]{Feehan_Maridakis_Lojasiewicz-Simon_coupled_Yang-Mills_v6} includes
an explanation of why Theorem \ref{mainthm:Lojasiewicz-Simon_gradient_inequality2} is so useful in applications to questions of global existence and convergence of gradient flows for energy functions on Banach spaces under the validity of the {\L}ojasiewicz--Simon gradient inequality.

\subsection{Notation and conventions}
\label{subsec:Notation}
For the notation of function spaces, we follow Adams and Fournier \cite{AdamsFournier}, and for functional analysis, Brezis \cite{Brezis} and Rudin \cite{Rudin}. As usual, we let $\NN:=\left\{0,1,2,3,\ldots\right\}$ denote the set of non-negative integers. We use $C=C(*,\ldots,*)$ to denote a constant which depends at most on the quantities appearing on the parentheses. In a given context, a constant denoted by $C$ may have different values depending on the same set of arguments and may increase from one inequality to the next. If $\sX, \sY$ is a pair of Banach spaces, then $\sL(\sX,\sY)$ denotes the Banach space of all continuous linear operators from $\sX$ to $\sY$. We denote the continuous dual space of $\sX$ by $\sX^* = \sL(\sX,\RR)$. We write $\alpha(x) = \langle x, \alpha \rangle_{\sX\times\sX^*}$ for the pairing between $\sX$ and its dual space, where $x \in \sX$ and $\alpha \in \sX^*$. If $T \in \sL(\sX, \sY)$, then its adjoint is denoted by $T^* \in \sL(\sY^*,\sX^*)$, where $(T^*\beta)(x) := \beta(Tx)$ for all $x \in \sX$ and $\beta \in \sY^*$.

\subsection{Acknowledgments}
\label{subsec:Acknowledgments}
Paul Feehan is very grateful to the Max Planck Institute for Mathematics, Bonn, and the Institute for Advanced Study, Princeton, for their support during the preparation of this article. He would like to thank Peter Tak{\'a}{\v{c}} for many helpful conversations regarding the {\L}ojasiewicz--Simon gradient inequality, for explaining his proof of \cite[Proposition 6.1]{Feireisl_Takac_2001} and how it can be generalized as described in this article, and for his kindness when hosting his visit to the Universit{\"a}t Rostock. He would also like to thank Brendan Owens for several useful conversations and his generosity when hosting his visit to the University of Glasgow. He thanks Haim Brezis for helpful comments on $L\log L$ spaces, Alessandro Carlotto for useful comments regarding the integrability of critical points of the Yamabe function, Sagun Chanillo for detailed and generous assistance with Hardy spaces, and Brendan Owens and Chris Woodward for helpful communications and comments regarding Morse--Bott theory. Both authors are very grateful to an anonymous referee for pointing out an error in an earlier statement and proof of Theorem \ref{mainthm:Lojasiewicz-Simon_gradient_inequality_energy_functional_Riemannian_manifolds} and to all editors and referees for their careful reading of our manuscript, corrections, and suggestions for improvements.

\section{{\L}ojasiewicz--Simon gradient inequalities for analytic and Morse--Bott functions}
\label{sec:Lojasiewicz-Simon_gradient_inequality_analytic_and_Morse-Bott_energy_functionals}
Our goal in this section is to prove the abstract {\L}ojasiewicz--Simon gradient inequalities for analytic and Morse--Bott functions stated in our Introduction, namely Theorems \ref{mainthm:Lojasiewicz-Simon_gradient_inequality_dualspace}, \ref{mainthm:Lojasiewicz-Simon_gradient_inequality}, \ref{mainthm:Lojasiewicz-Simon_gradient_inequality2}, and \ref{mainthm:Optimal_Lojasiewicz-Simon_gradient_inequality_Morse-Bott_energy_functional}. In Section \ref{subsec:Nonlinear_functional_analysis_preliminaries}, we review or establish some of the results in nonlinear functional analysis that we will subsequently require. As in Simon's original approach to the proof of his gradient inequality for analytic functions, we establish the result in infinite dimensions via a Lyapunov--Schmidt reduction to finite dimensions and an application of the finite-dimensional {\L}ojasiewicz gradient inequality, whose statement we recall in Section \ref{subsec:Finite-dimensional_Lojasiewicz-Simon_gradient_inequalities}. Sections \ref{subsec:Lojasiewicz-Simon_gradient_inequalities} and \ref{subsec:Lojasiewicz-Simon_gradient_inequalities_Hilbert_space} contain the proofs of the corresponding gradient inequalities for infinite-dimensional applications.

\subsection{Preliminaries on nonlinear functional analysis}
\label{subsec:Nonlinear_functional_analysis_preliminaries}
In this subsection, we gather a few elementary observations from nonlinear functional analysis that we will subsequently need.

\subsubsection{Smooth and analytic inverse and implicit function theorems for maps of Banach spaces}
\label{subsubsec:Smooth_and_analytic_inverse_and_implicit_function_theorems}
Statements and proofs of the Inverse Function Theorem for $C^k$ maps of Banach spaces are provided by Abraham, Marsden, and Ratiu \cite[Theorem 2.5.2]{AMR}, Deimling \cite[Theorem 4.15.2]{Deimling_1985}, and Zeidler \cite[Theorem 4.F]{Zeidler_nfaa_v1}; statements and proofs of the Inverse Function Theorem for \emph{analytic} maps of Banach spaces are provided by Berger \cite[Corollary 3.3.2]{Berger_1977} (complex), Deimling \cite[Theorem 4.15.3]{Deimling_1985} (real or complex), and Zeidler \cite[Corollary 4.37]{Zeidler_nfaa_v1} (real or complex). The corresponding $C^k$ or Analytic Implicit Function Theorems are proved in the standard way as corollaries, for example \cite[Theorem 2.5.7]{AMR} and \cite[Theorem 4.H]{Zeidler_nfaa_v1}.

\subsubsection{Differentiable and analytic maps on Banach spaces}
\label{subsubsec:Huang_2-1A}
We refer to \cite[Section 2.1A]{Huang_2006}; see also \cite[Section 2.3]{Berger_1977}. Let $\sX, \sY$ be a pair of Banach spaces, let $\sU\subset\sX$ be an open subset, and $\sF:\sU \to \sY$ be a map. Recall that $\sF$ is  \emph{Fr\'echet differentiable} at a point $x \in \sU$ with derivative $\sF'(x) \in \sL(\sX,\sY)$ if
\[
\lim_{y\to 0} \frac{1}{\|y\|_\sX}\|\sF(x + y) - \sF(x) - \sF'(x)y\|_\sY = 0.
\]
Recall from \cite[Definition 2.3.1]{Berger_1977}, \cite[Definition 15.1]{Deimling_1985}, \cite[Definition 8.8]{Zeidler_nfaa_v1} that $\sF$ is (real) \emph{analytic} at $x \in \sU$ if there exist a constant $r > 0$, a sequence of continuous symmetric $n$-linear forms $L_n:\otimes^n\sX \to \sY$ such that
\[
  \sum_{n\geq 1} \|L_n\| r^n < \infty,
\]
and a positive constant $\delta = \delta(x)$ such that
\begin{equation}
\label{Taylor_expansion}
\sF(x + y) = \sF(x) + \sum_{n\geq 1} L_n(y^n), \quad \|y\|_\sX < \delta,
\end{equation}
where $y^n \equiv (y,\ldots,y) \in \sX \times \cdots \times \sX$ ($n$-fold product). If $\sF$ is differentiable (respectively, analytic) at every point $x \in \sU$, then $\sF$ is differentiable (respectively, analytic) on $\sU$. It is a useful observation that if $\sF$ is analytic at $x\in\sX$, then it is analytic on a ball $B_x(\varepsilon)$ (see \cite[p. 1078]{Whittlesey_1965}).

\subsubsection{Gradient maps}
\label{subsubsec:Huang_2-1B}
We recall the following basic facts concerning gradient maps.

\begin{prop}[Properties of gradient maps]
\label{prop:Huang_2-1-2}
(See Huang \cite[Proposition 2.1.2]{Huang_2006}.)
Let $\sU$ be an open subset of a Banach space $\sX$, let $\sY$ be a Banach space that is continuously embedded in $\sX^*$, and let $\sM:\sU \to \sY \subset \sX^*$ be a continuous map. Then the following hold.
\begin{enumerate}
\item If $\sM$ is a gradient map for $\sE$, then
\[
\sE(x_1) - \sE(x_0) = \int_0^1 \langle x_1-x_0, \sM(tx_1 + (1-t)x_0) \rangle_{\sX\times\sX^*} \,dt, \quad\forall\, x_0, x_1 \in \sU.
\]

\item
\label{item:Huang_2-1-2_gradient_map_iff_symmetric_derivatives}
If $\sM$ is of class $C^1$, then $\sM$ is a gradient map if and only if all of its Fr\'echet derivatives, $\sM'(x)$ for $x \in \sU$, are symmetric in the sense that
\[
\langle w,\sM'(x)v \rangle_{\sX\times\sX^*} = \langle v,\sM'(x)w \rangle_{\sX\times\sX^*}, \quad\forall\, x \in \sU \text{ and } v,w \in \sX.
\]

\item
\label{item:Huang_2-1-2_analytic_gradient_map_implies_analytic_potential}
If $\sM$ is an analytic gradient map, then any potential function $\sE:\sU\to\RR$ for $\sM$ is analytic as well.
\end{enumerate}
\end{prop}

\begin{proof}
We prove Item \eqref{item:Huang_2-1-2_analytic_gradient_map_implies_analytic_potential} since this proof is omitted in \cite{Huang_2006}. Let $\imath: \sY \subset \sX^*$ denote the given continuous embedding. Because $\sM:\sU \to \sY$ is real analytic by hypothesis and the fact that the composition of a real analytic map with a bounded linear operator is real analytic, the derivative $\sE' = \imath \circ \sM: \sU\to \sX^*$ is real analytic as well. Hence, $\sE:\sU\to\RR$ is real analytic.
\end{proof}

\subsection{Finite-dimensional {\L}ojasiewicz and Simon gradient inequalities}
\label{subsec:Finite-dimensional_Lojasiewicz-Simon_gradient_inequalities}
We recall the finite-dimensional versions of the {\L}ojasiewicz--Simon gradient inequality.

\begin{thm}[Finite-dimensional {\L}ojasiewicz and Simon gradient inequalities]
\label{thm:Huang_2-3-1}
(See Huang \cite[Theorem 2.3.1]{Huang_2006}.)
%COMMENT The footnote seems unnecessary now
%\footnote{There is a typographical error in the statement of \cite[Theorem 2.3.1 (i)]{Huang_2006}, as Huang omits the hypothesis that $\sE'(z) = 0$; also our statement differs slightly from that of \cite[Theorem 2.3.1 (i)]{Huang_2006}, but is based on original sources.}
Let $U \subset \RR^n$ be an open subset, let $z \in U$ be a point, and let $\sE: U \to \RR$ be a function.
\begin{enumerate}
\item
\label{item:Huang_theorem_2-3-1_i}
If $\sE$ is real analytic on a neighborhood of $z$ and $\sE'(z) = 0$, then there exist constants $c\in (0,\infty)$ and $\theta \in (0,1)$ and $\sigma > 0$ such that
\begin{equation}
\label{eq:Lojasiewicz_1984_star}
|\sE'(x)| \geq c|\sE(x) - \sE(z)|^\theta,
\quad\forall\, x \in \RR^n \text{ such that } |x - z| < \sigma.
\end{equation}

\item
\label{item:Huang_theorem_2-3-1_ii}
Assume that $\sE$ is a $C^2$ function and $\sE'(z) = 0$. If the connected component $C$ of the critical point set $\{x \in U : \sE'(x) = 0\}$ that contains $z$ has the same dimension as the kernel of the Hessian $\sE''(z)$ near $z$, and $z$ lies in the interior of $C$, then there are positive constants $c$ and $\sigma$ such that
\begin{equation}
\label{eq:Simon_1996_lemma_1_page_80}
|\sE'(x)| \geq c|\sE(x) - \sE(z)|^{1/2},
\quad\forall\, x \in \RR^n  \text{ such that }  |x - z| < \sigma.
\end{equation}
\end{enumerate}
\end{thm}

Theorem \ref{thm:Huang_2-3-1} \eqref{item:Huang_theorem_2-3-1_i} is well known and was stated by {\L}ojasiewicz in \cite{Lojasiewicz_1963} and proved by him as \cite[Proposition 1, p. 92 (67)]{Lojasiewicz_1965} and Bierstone and Milman as \cite[Proposition 6.8]{BierstoneMilman}; see also the statements by Chill and Jendoubi \cite[Proposition 5.1 (i)]{Chill_Jendoubi_2003} and by {\L}ojasiewicz \cite[p. 1592]{Lojasiewicz_1993}.

Theorem \ref{thm:Huang_2-3-1} \eqref{item:Huang_theorem_2-3-1_ii} was proved (in certain Banach settings rather than just a Euclidean space setting) by Simon as \cite[Lemma 3.13.1]{Simon_1996} and Haraux and Jendoubi as \cite[Theorem 2.1]{Haraux_Jendoubi_2007}; see also the statement by Chill and Jendoubi \cite[Proposition 5.1 (ii)]{Chill_Jendoubi_2003}.

{\L}ojasiewicz used methods of \emph{semi-analytic sets} \cite{Lojasiewicz_1965} to prove Theorem \ref{thm:Huang_2-3-1} \eqref{item:Huang_theorem_2-3-1_i}. In general, so long as the constant $c$ is positive, its actual value is irrelevant to applications; the value of $\theta$ in the infinite-dimensional setting \cite[Theorem 2.4.2 (i)]{Huang_2006}, at least, is restricted to the range $[1/2, 1)$ and $\theta=1/2$ is optimal \cite[Theorem 2.7.1]{Huang_2006}.

\subsection{{\L}ojasiewicz--Simon gradient inequalities for analytic or Morse--Bott functions on Banach spaces}
\label{subsec:Lojasiewicz-Simon_gradient_inequalities}
We note that if $\sE : \sU \to \RR$ is a $C^2$ function on an open subset $\sU$ of a Banach space $\sX$, then its Hessian at a point $x_0 \in \sU$ is symmetric, that is
\begin{equation}
\label{eq:symmetry}
\langle x,\sE''(x_0)y\rangle_{\sX\times \sX^*} = \langle y, \sE''(x_0)x\rangle_{\sX\times \sX^*},
\end{equation}
for all $x, y\in \sX$; compare Proposition \ref{prop:Huang_2-1-2} \eqref{item:Huang_2-1-2_gradient_map_iff_symmetric_derivatives}.

Let $\tilde\sX$ and $\sX^*$ denote Banach spaces as in the statement of Theorem \ref{mainthm:Lojasiewicz-Simon_gradient_inequality} and let $K \subset \sX$ denote a finite-dimensional subspace. We shall identify $K$ with its images in $\tilde\sX$ and $\sX^*$. By \cite[Definition 4.20 and Lemma 4.21 (a)]{Rudin}, the subspace $K$ has a closed complement $\sY \subset \sX^*$ and there exists a continuous projection operator,
\begin{equation}
\label{eq:Projection_Xdual_onto_K}
\Pi : \sX^* \to K \subset \sX^*.
\end{equation}
The splitting $\sX^* = \sY \oplus K$ as a Banach space into closed subspaces induces corresponding splittings $\sX=\sX_0\oplus K$ and $\tilde\sX = \tilde\sX_0\oplus K$, where $\sX_0 := \sY\cap\sX$ and similarly for $\tilde\sX_0$. By restriction, the projection $\Pi:\sX^*\to \sX^*$ induces continuous projection operators with range $K$ on $\sX$ and $\tilde\sX$ that we continue to denote by $\Pi$. Hence, the projection $\Pi: \tilde\sX \to\tilde\sX$ restricts to a bounded linear operator $\Pi: \sX\to \tilde\sX$.

\begin{lem}[Properties of $C^2$ functions with Hessian operators that are Fredholm with index zero]
\label{lem:laux}
Assume the hypotheses of Theorem~\ref{mainthm:Lojasiewicz-Simon_gradient_inequality} and let $\Pi$ be as in \eqref{eq:Projection_Xdual_onto_K}, but now with $K = \Ker(\sE''(x_\infty):\sX\to\sX^*)$. Then there exist an open neighborhood $U_0 \subset \sU$ of $x_\infty$ and an open neighborhood $V_0\subset \tilde\sX$ of the origin such that the $C^1$ map,
\begin{equation}
\label{eq:Definition_Phi_map}
\Phi : \sX \supset \sU \ni x \mapsto \sM(x) +  \Pi(x-x_\infty) \in  \tilde\sX,
\end{equation}
when restricted to the set $U_0$, has a $C^1$ inverse $\Psi : V_0 \to U_0$. Moreover, there is a constant $C = C(\sM,U_0,V_0) \in [1,\infty)$ such that
\begin{equation}
\label{eq:aux}
\|\Psi(\Pi \alpha) - \Psi(\alpha)\|_{\sX} \leq C \|\sM(\Psi(\alpha))\|_{\tilde\sX},\quad \forall\, \alpha\in V_0.
\end{equation}
\end{lem}

\begin{proof}
Let $\zeta$ denote the embedding $\tilde\sX \subset \sX^*$ and observe that
  \[
    \sE' = \zeta\circ\sM \quad\text{and}\quad \sE'' = \zeta\circ \sM'.
  \]
  The derivative of $\Phi$ at $x_\infty$ is given by
  \[
    D\Phi(x_\infty) = \sM'(x_\infty) + \Pi: \sX\to \tilde\sX.
  \]
  If $D\Phi(x_\infty)x = 0$ for some $x \in \sX$, then $\sM'(x_\infty)x = -\Pi x \in K\subset \tilde\sX$. If $y\in K\subset \sX$, then
\begin{align*}
\langle y, \zeta\Pi x\rangle_{\sX\times \sX^*} &= - \langle y, \sE''(x_\infty)x \rangle_{\sX\times \sX^*}
\\
&= - \langle x , \sE''(x_\infty)y\rangle_{\sX\times \sX^*} \qquad\text{(by \eqref{eq:symmetry})}
\\
&= 0 \qquad (\text{since } y \in \Ker\sE''(x_\infty)).
\end{align*}
In particular, for $y = \Pi x  \in K\subset \sX$, recalling that $\jmath$ denotes the embedding $\sX\subset\sX^*$ and noting that $\jmath\Pi x = \zeta\Pi x \in \sX^*$, we have
\begin{equation}
\label{eq:inj}
\langle \Pi x, \jmath\Pi x\rangle_{\sX\times \sX^*} = \langle \Pi x, \zeta\Pi x\rangle_{\sX\times \sX^*} =0.
\end{equation}
Therefore, $\sE''(x_\infty)x = - \zeta \Pi x = 0$, by our hypothesis that the embedding $\jmath:\sX \to \sX^*$ is definite. Thus, $x \in \Ker \sE''(x_\infty) = K$ and because $\Pi x = 0$, we have $x=0$, that is, $D\Phi(x_\infty)$ has trivial kernel.

Because $\sM'(x_\infty)$ is Fredholm by hypothesis and $\Pi:\tilde\sX\to\tilde\sX$ has finite rank, it follows that
\[
D\Phi(x_\infty) = \sM'(x_\infty) + \Pi:\sX\to \tilde\sX
\]
is Fredholm, where $\Pi:\sX\to \tilde\sX$ denotes the composition of the embedding $\sX \subset \tilde\sX$ and the finite-rank projection $\Pi:\tilde\sX\to \tilde\sX$. Now $D\Phi(x_\infty) :\sX\to \tilde\sX$ is an injective Fredholm operator with index zero and therefore is surjective too. By the Open Mapping Theorem, $D\Phi(x_\infty)$ has a bounded inverse. By the Inverse Function Theorem for $\Phi$ near $x_\infty$, there exist an open neighborhood $U_1 \subset U$ of $x_\infty$ and a convex open neighborhood $V_1\subset \tilde\sX$ of the origin in $\tilde\sX$ so that the $C^1$ inverse $\Psi : V_1 \to U_1$ of $\Phi$ is well-defined. Since $\Pi: \tilde\sX \to \tilde\sX$ is bounded, we may choose $V_0 \subset V_1$, a smaller open neighborhood of the origin in $\tilde\sX$, with $\Pi(V_0) \subset V_1$ and set $U_0 := \Psi(V_0)$. From \eqref{eq:Definition_Phi_map}, we have
\[
\Phi(x) = \sM(x) + \Pi(x-x_\infty), \quad \forall\, x \in U_0,
\]
and the inverse function property and writing $\alpha = \Phi(x) \in V_0$ and $x = \Psi(\alpha)$ for $x \in U_0$, we obtain
\begin{equation}
\label{eq:IFT}
\alpha = \sM(\Psi(\alpha)) + \Pi(\Psi(\alpha) -x_\infty)),\quad \forall\, \alpha\in V_0.
\end{equation}
The Fundamental Theorem of Calculus then yields
\begin{align*}
\Psi(\Pi \alpha)-\Psi(\alpha)
&=
\int_{0}^{1} \left(\frac{d}{dt}\Psi(\alpha + t(\Pi \alpha-\alpha))\right)\,dt
\\
&= \left(\int_0^1 D\Psi(\alpha + t(\Pi \alpha-\alpha))\, dt\right)\, (\Pi \alpha -\alpha), \quad\forall\, \alpha \in V_0,
\end{align*}
where we use the fact that for $\alpha\in V_0$, we have $\alpha,\Pi \alpha \in V_1$ and, by convexity of $V_1$, the map $\Psi$ is well defined on the line segment joining $\alpha$ to $\Pi \alpha$. Therefore,
\[
\|\Psi(\Pi \alpha)-\Psi(\alpha)\|_{\sX}  \leq  M \|\Pi \alpha-\alpha\|_{\tilde\sX},  \quad\forall\, \alpha \in V_0,
\]
where, since $D\Psi(\alpha_1) \in \sL(\tilde\sX,\sX)$ is a continuous function of $\alpha_1 \in V_1$ (because $\Psi:V_1 \to U_1$ is $C^1$ by construction), we have
\[
M := \sup_{\alpha_1\in V_1} \|D\Psi(\alpha_1)\|_{\sL(\tilde\sX,\sX)} < \infty,
\]
since we may assume without loss of generality that $V_1 \supset V_0$ is a sufficiently small and bounded (convex) open neighborhood of the origin. Also, for all $\alpha \in V_0$,
\begin{align*}
\Pi \alpha-\alpha &= \Pi \alpha - \sM(\Psi(\alpha)) -  \Pi(\Psi(\alpha) -x_\infty)) \qquad\text{(by  \eqref{eq:IFT})}
\\
&= \Pi(\alpha - \Pi(\Psi(\alpha) -x_\infty)) - \sM(\Psi(\alpha)) \qquad\text{(since $\Pi^2 =\Pi$),}
\end{align*}
and
\begin{align*}
\|\Pi(\alpha- \Pi(\Psi(\alpha) - x_\infty))\|_{\tilde\sX} &\leq C_1\|\alpha- \Pi(\Psi(\alpha) - x_\infty)\|_{\tilde\sX}
\\
&= C_1\|\sM(\Psi(\alpha))\|_{\tilde\sX} \qquad \text{(by \eqref{eq:IFT}).}
\end{align*}
From the preceding expression for $\Pi \alpha-\alpha$ and preceding estimate, we conclude that
\[
\|\Pi \alpha-\alpha\|_{\tilde\sX} \leq (C_1+1)\|\sM(\Psi(\alpha))\|_{\tilde\sX}, \quad \forall\, \alpha\in V_0.
\]
Therefore, by combining the preceding inequalities, we obtain
\[
\|\Psi(\Pi \alpha)-\Psi(\alpha)\|_{\sX} \leq M(C_1+1)\|\sM(\Psi(\alpha))\|_{\tilde\sX}, \quad \forall\, \alpha\in V_0,
\]
and this concludes the proof of the assertions of Lemma \ref{lem:laux}.
\end{proof}

Recall the Definition~\ref{defn:Morse-Bott_function} of a Morse--Bott function $\sE$ and its set $\Crit \sE$ of critical values.

\begin{defn}[Lyapunov--Schmidt reduction of a $C^2$ function with Hessian operator that is Fredholm with index zero at a point]
\label{defn:LSreduction}
Assume the hypotheses of Theorem \ref{mainthm:Lojasiewicz-Simon_gradient_inequality} and let $\Psi: V_0 \cong U_0$ be the $C^1$ diffeomorphism of open neighborhoods, $V_0 \subset \tilde\sX$ of the origin and $U_0 \subset \sX$ of $x_\infty$, provided by Lemma~\ref{lem:laux}. We define the \emph{Lyapunov--Schmidt reduction of $\sE:U_0\to\RR$ at $x_\infty$} by
\[
\Gamma : K\cap V_0 \to \RR,\quad \alpha \mapsto \sE(\Psi(\alpha)),
\]
where $K = \Ker(\sE''(x_\infty):\sX\to\sX^*)$.
\end{defn}

Note that the origin in $\tilde\sX$ is a critical point of $\Gamma$ since $\Psi(0) = x_\infty$, the critical point of $\sE:\sU \to \RR$ in Lemma~\ref{lem:laux}, and
\[
\Gamma'(0)x = \sE'(\Psi(0))D\Psi(0)x = \sE'(x_\infty)D\Psi(0)x = 0, \quad\forall\, x \in \sX.
\]
The following lemma plays a crucial role in the proofs of Theorems \ref{mainthm:Lojasiewicz-Simon_gradient_inequality} and \ref{mainthm:Optimal_Lojasiewicz-Simon_gradient_inequality_Morse-Bott_energy_functional}.

\begin{lem}[Properties of the Lyapunov--Schmidt reduction of a $C^2$ function]
\label{lem:Lyapunov}
Assume the hypotheses of Theorem \ref{mainthm:Lojasiewicz-Simon_gradient_inequality} together with the notation of Lemma~\ref{lem:laux} and Definition \ref{defn:LSreduction}.
\begin{enumerate}
\item \label{eq:Lj1} If $\sE$ is Morse--Bott at $x_\infty$, then there is an open neighborhood $\sV$ of the origin in $K\cap V_0$ where the Lyapunov--Schmidt reduction of $\sE$ is a constant function, that is, $\Gamma \equiv \sE(x_\infty)$ on $\sV$.

\item \label{eq:Lj2} If $\sM$ is real analytic on $\sU$, then $\Gamma$ is real analytic on $K\cap V_0$.
\end{enumerate}
\end{lem}

\begin{rmk}[Relationship between the Morse--Bott and other integrability conditions]
Item \eqref{eq:Lj1} in Lemma \ref{lem:Lyapunov} is closely related to \cite[Lemma 1]{Adams_Simon_1988} due to Adams and Simon, which asserts (in our notation) that $\Gamma \equiv \Gamma(0)$ on an open neighborhood of the origin in $K$ if and only if the following integrability condition holds:
\begin{multline}
\label{eq:Star}
\tag{$\star$}  
\forall\, v \in K, \ \exists\, u \in C^0((0,1);\tilde{\sX}) \text{ such that } O(u) \subset \Crit\sE
\\
\text{and } \lim_{t\downarrow 0} u(t) = 0 \text{ (in $\tilde{\sX}$) } \text{ and } \lim_{t\downarrow 0} u(t)/t = v \text{ (in $\tilde{\sG}$)},  
\end{multline}
where $O(u) := \{u(t): t \in (0,1)\}$ and $\tilde\sG$ is a Banach space with continuous embeddings $\tilde{\sX} \subset \tilde{\sG} \subset \sX^*$, as in the hypotheses of Theorem \ref{mainthm:Lojasiewicz-Simon_gradient_inequality2}. (Adams and Simon choose $\tilde{\sG}$ to be a certain Hilbert space but do not otherwise precisely specify the regularity properties of the path $u$ in their definition.) See Feehan \cite[Appendix C]{Feehan_lojasiewicz_inequality_all_dimensions_morse-bott} for further discussion. 
%COMMENT Beware that our discussion in Appendix C of Feehan_lojasiewicz_inequality_all_dimensions_morse-bott is informal.
\end{rmk}

\begin{proof}[Proof of Lemma \ref{lem:Lyapunov}]
If $\sE$ is Morse--Bott at $x_\infty$ then, by shrinking $U_0$ if necessary, we may assume that the set $\Crit \sE\cap U_0$ is a submanifold of $U_0$ with tangent space
\[
T_{x_\infty} \Crit \sE = K = \Ker\left(\sM'(x_\infty) :\sX\to \tilde\sX\right).
\]
Then the restriction of the map $\Phi:U_0\to V_0$ in \eqref{eq:Definition_Phi_map},
\begin{equation}
\label{eq:Phi_restricted_CritE}
\Phi: \Crit\sE\cap U_0\to K\cap V_0,
\end{equation}
has derivative at $x_\infty$ given by
\[
D\Phi(x_\infty) = \sM'(x_\infty) + \Pi = \Pi : K \to K.
\]
The preceding operator comprises the embedding $\eps:\sX\to \tilde\sX$
restricted to $K$ and resulting isomorphism from $K\subset \sX$ to $K\subset \tilde\sX$. An application of the Inverse Function Theorem shows that the inverse of the map \eqref{eq:Phi_restricted_CritE} is defined in a neighborhood $\sV$ of the origin in $K\cap V_0$ and is the restriction of the map $\Psi:V_0\to U_0$ to $K\cap V_0$. Therefore, $\Psi(\sV)\subset \Crit \sE \cap U_0$ and we compute
\[
\Gamma'(\alpha) = \sE'(\Psi(\alpha)) D\Psi(\alpha) = 0,\quad \forall\, \alpha\in \sV.
\]
Therefore, $\Gamma(\alpha) = \Gamma(0) = \sE(x_\infty)$, for every $\alpha\in \sV$. This proves Item \eqref{eq:Lj1}.

To prove Item \eqref{eq:Lj2}, we recall from Lemma~\ref{lem:laux} that
the map $\Psi: V_0 \to U_0$ is a $C^1$ diffeomorphism. Moreover,
$\Phi$ is real analytic since $\sM$ is real analytic by hypothesis and
by the definition \eqref{eq:Definition_Phi_map} of $\Phi$.  The
Analytic Inverse Function Theorem (see Section
\ref{subsubsec:Smooth_and_analytic_inverse_and_implicit_function_theorems}) implies that the inverse map $\Psi: V_0 \to U_0$ is also real analytic and therefore its restriction to the intersection $K\cap V_0$ of a finite-dimensional linear subspace $K \subset\tilde\sX$ with the open set $V_0 \subset \tilde\sX$ is still real analytic.
Because the gradient map $\sM:\sU\to\tilde\sX$ is real analytic, it follows that its
potential function $\sE : \sU\to \RR$ is real analytic
by Proposition \ref{prop:Huang_2-1-2}
\eqref{item:Huang_2-1-2_analytic_gradient_map_implies_analytic_potential}.
Therefore, the composition $\Gamma = \sE\circ\Psi:K\cap V_0 \to \RR$ is a real analytic function.
\end{proof}

We then have the

\begin{prop}[{\L}ojasiewicz--Simon gradient inequalities for analytic and Morse--Bott functions on Banach spaces]
\label{prop:LSprop}
Assume the hypotheses of Lemma~\ref{lem:laux}. Then the following hold.
\begin{enumerate}
\item
\label{item:LSProp_Morse-Bott}
If $\sE$ is Morse--Bott at $x_\infty$, then there exist an open neighborhood $W_0 \subset \sU$ of $x_\infty$ and a constant $C = C(\sE,W_0) \in [1, \infty)$ such that
\[
|\sE(x) - \sE(x_\infty)| \leq C\|\sM(x)\|_{\tilde\sX}^2, \quad \forall\, x\in W_0.
\]
\item
\label{item:LSProp_analytic}
If $\sM$ is analytic on $\sU$, then there exist an open neighborhood $W_0 \subset \sU$ of $x_\infty$ and constants $C = C(\sE,W_0) \in [1, \infty)$ and $\beta\in (1,2]$ such that
\[
|\sE(x) - \sE(x_\infty)| \leq C\|\sM(x)\|_{\tilde\sX}^\beta, \quad \forall\, x\in W_0.
\]
\end{enumerate}
\end{prop}

\begin{proof}
Denote $x=\Psi(\alpha) \in U_0$ for $\alpha \in V_0$ and recall the definitions of the open neighborhoods $U_1$ and $V_1$ from the proof of Lemma~\ref{lem:laux}. By shrinking $U_1$ if necessary, we may assume that $U_1$ is contained in a bounded, convex, open subset $U_2 \subset \sU$. For $\alpha\in V_0$ we have $\alpha, \Pi \alpha \in V_1$ (as in the proof of Lemma~\ref{lem:laux}) and therefore $\Psi(\alpha), \Psi(\Pi \alpha) \in U_0$ and the line segment joining $\Psi(\alpha)$ to $\Psi(\Pi \alpha)$ lies in $U_2$. The Definition \ref{defn:LSreduction} of $\Gamma$, the fact that
\[
\Pi \alpha \in K\cap V_0, \quad\forall\, \alpha \in V_0
\]
and the Fundamental Theorem of Calculus then give
\begin{align*}
\sE(\Psi(\alpha)) - \Gamma(\Pi \alpha) &= \sE(\Psi(\alpha)) - \sE(\Psi(\Pi \alpha))
\\
&= -\int_{0}^{1} \frac{d}{dt}\sE\left(\Psi(\alpha) + t(\Psi(\Pi \alpha)-\Psi(\alpha))\right)\,dt, \quad\forall\, \alpha \in V_0,
\end{align*}
and thus
\begin{multline}
\label{eq:ftc}
\sE(\Psi(\alpha)) - \Gamma(\Pi \alpha)
\\
= \left(-\int_0^1 \sE'(\Psi(\alpha) + t(\Psi(\Pi \alpha)-\Psi(\alpha)))\, dt\right)\,  (\Psi(\Pi \alpha)-\Psi(\alpha)),
\quad\forall\, \alpha \in V_0.
\end{multline}
Note that
\begin{multline*}
\|\sE'(\Psi(\alpha) + t(\Psi(\Pi \alpha)- \Psi(\alpha)))\|_{\sX^*}
\\
\leq \|\sE'(\Psi(\alpha) + t (\Psi(\Pi \alpha)-\Psi(\alpha))) - \sE'(\Psi(\alpha))\|_{\sX^*}
+ \|\sE'(\Psi(\alpha))\|_{\sX^*}, \quad\forall\, \alpha \in V_0,
\end{multline*}
and therefore,
\begin{multline}
\label{eq:Eprime_Psi_triangle}
\|\sE'(\Psi(\alpha) + t(\Psi(\Pi \alpha)- \Psi(\alpha)))\|_{\sX^*}
\\
\leq C_0\|\sM(\Psi(\alpha) + t (\Psi(\Pi \alpha)-\Psi(\alpha))) - \sM(\Psi(\alpha))\|_{\tilde\sX}
+ C_0\|\sM(\Psi(\alpha))\|_{\tilde\sX},  \quad\forall\, \alpha \in V_0,
\end{multline}
where $C_0 \in [1,\infty)$ is the norm of the embedding $\zeta: \tilde\sX\hookrightarrow \sX^*$. Similarly, the Fundamental Theorem of Calculus yields
\begin{align*}
{}&\sM(\Psi(\alpha) + t (\Psi(\Pi \alpha)-\Psi(\alpha))) - \sM(\Psi(\alpha))
\\
&= \int_{0}^{1} \frac{d}{ds}\sM\left(\Psi(\alpha) + st (\Psi(\Pi \alpha)-\Psi(\alpha))\right)\,ds
\\
&= t\left(\int_0^1 \sM' (\Psi(\alpha) + st (\Psi(\Pi \alpha)-\Psi(\alpha)))\, ds\right)(\Psi(\Pi \alpha)-\Psi(\alpha)), \quad\forall\, \alpha \in V_0.
\end{align*}
Thus, by taking norms of the preceding equality we obtain
\begin{equation}
\label{eq:xua}
\|\sM(\Psi(\alpha) + t (\Psi(\Pi \alpha )-\Psi(\alpha))) - \sM(\Psi(\alpha))\|_{\tilde\sX} \leq M_1 \|\Psi(\Pi \alpha)-\Psi(\alpha)\|_\sX, \quad\forall\, \alpha \in V_0,
\end{equation}
where, since $\sM:\sU\to \tilde\sX$ is $C^1$ by hypothesis, we have
\[
M_1 := \sup_{x\in U_2} \|\sM'(x)\|_{\sL(\sX,\tilde\sX)} < \infty,
\]
because we may assume (by further shrinking $U_1$ if necessary) that $U_2\subset U$ is a sufficiently small and bounded, convex, open neighborhood of $x_\infty$.

Combining the inequalities \eqref{eq:Eprime_Psi_triangle} and \eqref{eq:xua} with the equality \eqref{eq:ftc} yields
\[
|\sE(\Psi(\alpha)) - \Gamma(\Pi \alpha)|
\leq
C_0\left(M_1\|\Psi(\Pi \alpha)-\Psi(\alpha)\|_\sX +  \|\sM(\Psi(\alpha))\|_{\tilde\sX} \right)
\|\Psi(\Pi \alpha)-\Psi(\alpha)\|_\sX,
\]
and so combining the preceding inequality with \eqref{eq:aux} gives
\begin{equation}
\label{eq:EPsi_minus_GammaPi_bounded_by_gradient_E_Psi_squared}
|\sE(\Psi(\alpha)) - \Gamma(\Pi \alpha)| \leq C\|\sM(\Psi(\alpha))\|^2_{\tilde\sX}, \quad\forall\, \alpha \in V_0.
\end{equation}
We now invoke the hypotheses that $\sE$ is Morse--Bott at $x_\infty$ or analytic near $x_\infty$.

When $\sE$ is Morse--Bott at $x_\infty$, Lemma~\ref{lem:Lyapunov}~\eqref{eq:Lj1} provides an open neighborhood $\sV$ of the origin in
$K \cap V_0$ such that $\Gamma \equiv \sE(x_\infty)$ on $\sV$. By choosing $W_0 =  \Psi(V_0\cap\Pi^{-1}(\sV))$ and noting that $\Pi:\tilde\sX\to\tilde\sX$ is a continuous (linear) map, we obtain from \eqref{eq:EPsi_minus_GammaPi_bounded_by_gradient_E_Psi_squared} that
\[
|\sE(x) - \sE(x_\infty)| \leq C  \|\sM(x)\|_{\tilde\sX}^2, \quad\forall\, x = \Psi(\alpha)\in W_0,
\]
which proves Item \eqref{item:LSProp_Morse-Bott}.

Finally, when $\sE$ is analytic on $\sU$ then Lemma~\ref{lem:Lyapunov}~\eqref{eq:Lj2} implies that $\Gamma$ is analytic on $K\cap V_0$. The finite-dimensional {\L}ojasiewicz gradient inequality \eqref{eq:Lojasiewicz_1984_star} in Theorem \ref{thm:Huang_2-3-1} \eqref{item:Huang_theorem_2-3-1_i} applies to give, for a possibly smaller neighborhood $V_2\subset V_0$ of the origin, constants $C \in [1,\infty)$ and $\beta\in(1,2]$ such that
\begin{equation}
\label{eq:finiteLS}
|\Gamma(\Pi \alpha) - \sE(x_\infty)| \leq C \|\Gamma'(\alpha)\|_{K^*}^\beta, \quad\forall\, \alpha \in V_2.
\end{equation}
But $\Gamma'(\Pi \alpha) = \sE'(\Psi(\Pi \alpha)) D\Psi(\Pi \alpha)$ by Definition~\ref{defn:LSreduction} of $\Gamma$ and thus
\begin{equation}
\label{eq:Gradient_Gamma_bounded_by_gradient_E}
\|\Gamma'(\Pi \alpha)\|_{K^*}
\leq
M_2\|\sE'(\Psi(\Pi \alpha))\|_{\sX^*}
\leq
C_0M_2\|\sM(\Psi(\Pi \alpha))\|_{\tilde\sX}, \quad\forall\, \alpha \in V_2.
\end{equation}
Here, since $D\Psi(\alpha_1) \in \sL(\tilde\sX,\sX)$ is a continuous function of $\alpha_1 \in V_1$ (because $\Psi:V_1 \to U_1$ is $C^1$ by construction), we have
\[
M_2 := \sup_{\alpha_1\in V_1} \|D\Psi(\alpha_1)\|_{\sL(\tilde\sX,\sX)} < \infty.
\]
The constant $M_2$ is finite because we may assume without loss of generality that $V_1 \supset V_2$ is a sufficiently small and bounded, convex, open neighborhood of the origin. Hence, for every $\alpha\in V_2$,
\begin{align*}
|\Gamma(\Pi \alpha) - \sE(x_\infty)|
&\leq
C\|\sM(\Psi(\Pi \alpha))\|_{\tilde\sX}^\beta \qquad \text{(by \eqref{eq:finiteLS} and \eqref{eq:Gradient_Gamma_bounded_by_gradient_E})}
\\
&\leq C\left(\|\sM(\Psi(\Pi \alpha)) - \sM(\Psi(\alpha))\|_{\tilde\sX} + \|\sM(\Psi(\alpha))\|_{\tilde\sX} \right)^\beta
\\
&\leq
C\left(\|\Psi(\Pi \alpha)-\Psi(\alpha)\|_\sX + \|\sM(\Psi(\alpha))\|_{\tilde\sX}\right)^\beta \qquad\text{(by \eqref{eq:xua} for $t=1$).}
\end{align*}
By combining the preceding inequality with \eqref{eq:aux}, we obtain
\begin{equation}
\label{eq:Gamma_minus_E_bounded_by_gradient_E_Psi_alpha}
|\Gamma(\Pi \alpha) - \sE(x_\infty)| \leq C\|\sM(\Psi(\alpha))\|_{\tilde\sX}^\beta, \quad\forall\, \alpha\in V_2.
\end{equation}
Consequently, for every $\alpha\in V_2$,
\begin{align*}
|\sE(\Psi(\alpha)) - \sE(x_\infty)|
&\leq
|\sE(\Psi(\alpha)) - \Gamma(\Pi \alpha)| + | \Gamma(\Pi \alpha) - \sE(x_\infty)|
\\
&\leq C\left( \|\sM(\Psi(\alpha))\|_{\tilde\sX}^2 + \|\sM(\Psi(\alpha))\|_{\tilde\sX}^\beta \right)
\qquad\text{(by \eqref{eq:EPsi_minus_GammaPi_bounded_by_gradient_E_Psi_squared} and
\eqref{eq:Gamma_minus_E_bounded_by_gradient_E_Psi_alpha})}
\\
&\leq C \|\sM(\Psi(\alpha))\|_{\tilde\sX}^\beta \left( 1+ \|\sM(\Psi(\alpha))\|_{\tilde\sX}^{2-\beta} \right)
\\
&\leq CM_3 \|\sM(\Psi(\alpha))\|_{\tilde\sX}^\beta.
\end{align*}
Here, for small enough $V_2$ and noting that $\sM(\Psi(\alpha)) \in \tilde\sX$ is a continuous function of $\alpha \in V_2$ (since $\Psi:V_1\to U_1$ is $C^1$ by construction), we have
\[
M_3 := 1+ \sup_{\alpha \in V_2} \|\sM(\Psi(\alpha))\|_{\tilde\sX}^{2-\beta} < \infty.
\]
Setting $x = \Psi(\alpha)$ for $\alpha \in V_2$ yields
\[
|\sE(x) - \sE(x_\infty)| \leq CM_3 \|\sM(x)\|_{\tilde\sX}^\beta, \quad\forall\, x \in \Psi(V_2).
\]
We now choose $W_0 = \Psi(V_2)$ to complete the proof of Item \eqref{item:LSProp_analytic} and hence Proposition~\ref{prop:LSprop}.
\end{proof}

We can now complete the

\begin{proof}[Proofs of
  Theorems~\ref{mainthm:Lojasiewicz-Simon_gradient_inequality} and
  \ref{mainthm:Optimal_Lojasiewicz-Simon_gradient_inequality_Morse-Bott_energy_functional}]
  The conclusions follow immediately from Proposition~\ref{prop:LSprop}.
\end{proof}

\subsection{Generalized {\L}ojasiewicz--Simon gradient inequalities for analytic functions or Morse--Bott functions on Banach spaces and gradient maps valued in Hilbert spaces}
\label{subsec:Lojasiewicz-Simon_gradient_inequalities_Hilbert_space}
In this section, we complete the proofs of Theorems \ref{mainthm:Lojasiewicz-Simon_gradient_inequality2} and
  \ref{mainthm:Optimal_Lojasiewicz-Simon_gradient_inequality_Morse-Bott_energy_functional}.
Let $\sX, \tilde\sX, \sG, \tilde\sG$, and $\sX^*$ denote Banach spaces as in the statement of Theorem \ref{mainthm:Lojasiewicz-Simon_gradient_inequality2} and let $K \subset \sX$ denote a finite-dimensional subspace. We shall identify $K$ with its images in $\sX, \sG, \tilde\sX, \tilde\sG$, and $\sX^*$. By \cite[Definition 4.20 and Lemma 4.21 (a)]{Rudin}, the subspace $K$ has a closed complement $\sY \subset \sX^*$ and there exists a continuous projection operator,
\begin{equation}
\label{eq:Projection_Xdual_onto_K2}
\Pi : \sX^* \to K \subset \sX^*.
\end{equation}
The splitting $\sX^* = \sY \oplus K$ as a Banach space into closed subspaces induces corresponding splittings,
\[
  \sX=\sX_0\oplus K, \quad \sG=\sG_0\oplus K, \quad \tilde\sX= \tilde\sX_0\oplus K, \quad\text{and}\quad \tilde\sG = \tilde\sG_0\oplus K,
\]
where $\tilde\sG_0 := \sY\cap \tilde\sG$ and similarly for the remaining closed complements. By restriction, the projection $\Pi:\sX^*\to \sX^*$ induces continuous projection operators with range $K$ on $\sX, \tilde\sX, \sG$, and $\tilde\sX$ that we continue to denote by $\Pi$.

Because the compositions of embeddings,
\[
K\subset \sX\subset \sG\subset \tilde\sG \quad \text{and}\quad K\subset \sX\subset \tilde\sX\subset \tilde\sG,
\]
are equal by hypothesis, it follows that the projection $\Pi: \tilde\sG \to\tilde\sG$ restricts to bounded linear operators $\Pi: \sX\to \tilde\sX$ and $\Pi: \sG \to K\subset \tilde\sG$.

Recall from the proof of Lemma~\ref{lem:laux} that there exist an open neighborhood $U_1 \subset \sU$ of $x_\infty$ and a convex open neighborhood $V_1\subset \tilde\sX$ of the origin in $\tilde\sX$ such that the restriction of the map $\Phi:\sX \supset \sU \to \tilde\sX$, first introduced in \eqref{eq:Definition_Phi_map} (and recalled in the forthcoming Lemma \ref{lem:laux2}), to $\Phi:U_1 \to V_1$ has a $C^1$ inverse $\Psi: V_1 \to U_1$. We then have the

%COMMENT-PF-10-14-2016 Better lemma title?
\begin{lem}[Continuous extension]
\label{lem:inv_ext}
Assume the hypotheses of Theorem~\ref{mainthm:Lojasiewicz-Simon_gradient_inequality2} and define a map $\sT:\sU \to \sL(\sG, \tilde\sG)$ by
\begin{equation}
\label{eq:Definition_sT_map}
\sT : \sX \supset \sU \ni x\mapsto \sM_1(x) +  \Pi \in \sL(\sG, \tilde\sG).
\end{equation}
Then the following hold.
\begin{enumerate}
\item \label{item:lemma_inv_ext_DPhi}
For every $x \in \sU$, the bounded linear operator $\sT(x) : \sG \to \tilde\sG$ is a continuous extension of $D\Phi(x): \sX \to \tilde\sX$.

\item \label{item:lemma_inv_ext_DPsi}
The neighborhoods $U_1 \subset \sU$ of $x_\infty$ and $V_1 = \Phi(U_1) \subset \tilde\sX$ of the origin can be chosen such that for every $x\in U_1$, the inverse operator $\sT(x)^{-1} : \tilde\sG \to \sG$ is well-defined and is a bounded extension of $D\Psi(\Phi(x))$.

\item \label{item:lemma_inv_ext_map_continuity}
The map $U_1 \ni x \mapsto \sT(x)^{-1} \in \sL(\tilde\sG,\sG)$ is continuous.
\end{enumerate}
\end{lem}

\begin{proof}
Consider Item \eqref{item:lemma_inv_ext_DPhi}. By hypothesis, $\sM_1(x) \in \sL(\sG, \tilde\sG)$ is a continuous extension of $\sM'(x) \in \sL(\sX, \tilde\sX)$ for each $x \in \sU$ and thus Item \eqref{item:lemma_inv_ext_DPhi} follows by the definition \eqref{eq:Definition_Phi_map} of $\Phi$, giving $D\Phi(x) = \sM'(x) + \Pi$, and the definition \eqref{eq:Definition_sT_map} of $\sT(x)$.

Consider Item \eqref{item:lemma_inv_ext_DPsi}. By hypothesis, the operator $\sM_1(x_\infty): \sG\to \tilde\sG$ is Fredholm of index zero. Hence, by applying the same argument as in the proof of Lemma~\ref{lem:laux}, but with $\sG$ and $\tilde\sG$ in place of $\sX$ and $\tilde\sX$, we see that the bounded linear operator $\sT(x_\infty) : \sG\to \tilde\sG$ has a bounded inverse $\sT(x_\infty)^{-1}: \tilde\sG\to \sG$. The subset of \emph{invertible} linear operators $\sI(\sG, \tilde\sG)$ is open in $\sL(\sG, \tilde\sG)$. By hypothesis, the map $\sT:\sU \to \sL(\sG, \tilde\sG)$ is continuous and therefore we may choose the neighborhood $U_1$ of $x_\infty\in \sX$ (and $V_1$ of the origin in $\tilde\sX$) small enough that $\sT(U_1)$ is contained in the subset $\sI(\sG, \tilde\sG)$ of invertible operators. Hence, for every $x\in U_1$, the bounded linear operator $D\Phi(x):\sX\to\tilde\sX$ has a  bounded extension $\sT(x):\sG\to\tilde\sG$ that is invertible. Therefore,
\[
\sT(x) \restriction \sX = D\Phi(x)
\quad\text{and}\quad
\sT(x)^{-1} \restriction \tilde\sX = (D\Phi(x))^{-1} = D\Psi(\Phi(x)),
\]
for every $x\in U_1$; thus, $\sT(x)^{-1}:\tilde\sG \to \sG$ is a bounded extension of $D\Psi(\Phi(x)):\tilde\sX\to\sX$. This establishes Item \eqref{item:lemma_inv_ext_DPsi}.

Consider Item \eqref{item:lemma_inv_ext_map_continuity}. The inversion map
\[
\iota:\sI(\sG, \tilde\sG) \ni T \mapsto T^{-1} \in \sI(\sG, \tilde\sG)
\]
is continuous and hence the composition
\[
U_1 \ni x \mapsto (\iota\circ \sT)(x) = \sT(x)^{-1} \in \sI(\sG, \tilde\sG)
\]
is also continuous. This establishes Item \eqref{item:lemma_inv_ext_map_continuity} and completes the proof of Lemma \ref{lem:inv_ext}.
\end{proof}

We then have the following variant of Lemma \ref{lem:laux}.

\begin{lem}[Properties of $C^2$ functions with Hessian operators that are Fredholm with index zero]
\label{lem:laux2}
Assume the hypotheses of Theorem~\ref{mainthm:Lojasiewicz-Simon_gradient_inequality2} and let $\Pi$ be as in \eqref{eq:Projection_Xdual_onto_K2}, now with $K = \Ker(\sE''(x_\infty):\sX\to\sX^*)$. Then there exist an open neighborhood $U_0 \subset \sU$ of $x_\infty$ and an open neighborhood $V_0\subset \tilde\sX$ of the origin such that the $C^1$ map $\Phi$ in \eqref{eq:Definition_Phi_map}, namely
\[
\Phi : \sU\to  \tilde\sX,\quad  x \mapsto \sM(x) +  \Pi(x-x_\infty),
\]
when restricted to the set $U_0$, has a $C^1$ inverse $\Psi : V_0 \to U_0$. Moreover, there is a constant $C = C(\sM,U_0,V_0) \in [1,\infty)$ such that
\begin{equation}
\label{eq:aux2}
\|\Psi(\Pi \alpha) - \Psi(\alpha)\|_{\sG} \leq C \|\sM(\Psi(\alpha))\|_{\tilde\sG},\quad \forall\, \alpha\in V_0.
\end{equation}
\end{lem}

\begin{proof}
Let $\Psi:V_1\to U_1$ be the $C^1$ inverse of the map $\Phi:U_1\to V_1$ defined in the proof of Lemma \ref{lem:laux}, now for the possibly smaller open neighborhoods $U_1 \subset \sU$ and $V_1 \subset \tilde\sX$ provided by Lemma~\ref{lem:inv_ext}. By shrinking the neighborhoods $U_1$ and $V_1$ further if necessary, we may again assume that $V_1$ is a convex neighborhood of the origin in $\tilde\sX$. Since $\Pi: \tilde\sX \to \tilde\sX$ is a bounded linear operator, we may choose a smaller open neighborhood $V_0 \subset V_1$ of the origin in $\tilde\sX$ with $\Pi(V_0) \subset V_1$ and define $U_0 := \Psi(V_0)$.

It remains to verify the inequality \eqref{eq:aux2} by making the necessary changes to the verification of the inequality \eqref{eq:aux} in Lemma~\ref{lem:laux}. By Lemma~\ref{lem:inv_ext}, the following map is well-defined,
\[
\hat{\sT} : V_0 \ni \alpha \mapsto \sT(\Psi(\alpha))^{-1} \in \sL( \tilde \sG, \sG).
\]
We first observe that
\begin{align*}
\Psi(\Pi \alpha)-\Psi(\alpha)
&=
\int_{0}^{1}\frac{d}{dt}\Psi(\alpha + t(\Pi \alpha-\alpha))\,dt
\\
&= \left(\int_0^1 D\Psi(\alpha + t(\Pi \alpha-\alpha))\, dt\right)\, (\Pi \alpha -\alpha)
\\
&= \left(\int_0^1 \hat{\sT}(\alpha + t(\Pi \alpha-\alpha))\, dt\right)\, (\Pi \alpha -\alpha), \quad\forall\, \alpha \in V_0,
\end{align*}
since $\alpha + t(\Pi \alpha-\alpha) \in V_1 = \Phi(U_1) \subset \tilde\sX$ for all $t\in[0,1]$ and so Lemma~\ref{lem:inv_ext} gives
\[
\hat{\sT}(\alpha + t(\Pi \alpha-\alpha))
=
D\Psi(\alpha + t(\Pi \alpha-\alpha))
\in
\sL(\tilde\sG,\sG).
\]
Thus, we have
\[
\|\Psi(\Pi \alpha)-\Psi(\alpha)\|_{\sG}  \leq  M \|\Pi \alpha-\alpha\|_{\tilde\sG},  \quad\forall\, \alpha \in V_0,
\]
where
\[
M := \sup_{\alpha_1\in V_1} \|\hat{\sT}(\alpha_1)\|_{\sL(\tilde\sG,\sG)} < \infty,
\]
and $M$ is finite (possibly after shrinking $V_1$) by Lemma \ref{lem:inv_ext} \eqref{item:lemma_inv_ext_map_continuity}, which provides continuity of $\hat{\sT}$. Finally, for all $\alpha \in V_0$,
\begin{align*}
\Pi \alpha-\alpha &= \Pi \alpha - \Phi(\Psi(\alpha))
\quad\text{(since $\Phi(\Psi(\alpha)=\alpha$)}
\\
&= \Pi \alpha - \sM(\Psi(\alpha)) -  \Pi(\Psi(\alpha) -x_\infty))
\quad\text{(by the definition \eqref{eq:Definition_Phi_map} of $\Phi$)}
\\
&= \Pi(\alpha - \Pi(\Psi(\alpha) -x_\infty)) - \sM(\Psi(\alpha)) \qquad\text{(since $\Pi^2 =\Pi$),}
\end{align*}
and, if $C_1 = C_1(K) \in [1,\infty)$ is the norm of the projection operator $\Pi \in \sL(\tilde\sG)$, then
\begin{align*}
\|\Pi(\alpha- \Pi(\Psi(\alpha) - x_\infty))\|_{\tilde\sG} &\leq C_1\|\alpha- \Pi(\Psi(\alpha) - x_\infty)\|_{\tilde\sG}
\\
&= C_1\|\sM(\Psi(\alpha))\|_{\tilde\sG}.
\end{align*}
Taking $\tilde\sG$ norms of the preceding identity, we conclude that
\[
\|\Pi \alpha-\alpha\|_{\tilde\sG} \leq (C_1+1)\|\sM(\Psi(\alpha))\|_{\tilde\sG}, \quad \forall\, \alpha\in V_0.
\]
Therefore, by combining the preceding inequalities, we obtain
\[
\|\Psi(\Pi \alpha)-\Psi(\alpha)\|_{\sG} \leq M(C_1+1)\|\sM(\Psi(\alpha))\|_{\tilde\sG}, \quad \forall\, \alpha\in V_0,
\]
which is the desired inequality \eqref{eq:aux2}. This concludes the proof of Lemma \ref{lem:laux2}.
\end{proof}

Next, we have the following variant of Proposition \ref{prop:LSprop}.

\begin{prop}[{\L}ojasiewicz--Simon gradient inequalities for analytic and Morse--Bott functions on Banach spaces]
\label{prop:LSprop2}
Assume the hypotheses of Lemma~\ref{lem:laux2}. Then the following hold.
\begin{enumerate}
\item
\label{item:LSProp_Morse-Bott2}
If $\sE$ is Morse--Bott at $x_\infty$, then there exist an open neighborhood $W_0 \subset \sU$ of $x_\infty$ and a constant $C = C(\sE,W_0) \in [1, \infty)$ such that
\[
|\sE(x) - \sE(x_\infty)| \leq C\|\sM(x)\|_{\tilde\sG}^2, \quad \forall\, x\in W_0.
\]
\item
\label{item:LSProp_analytic2}
If $\sM$ is analytic on $\sU$, then there exist an open neighborhood $W_0 \subset \sU$ of $x_\infty$ and constants $C = C(\sE,W_0) \in [1, \infty)$ and $\beta\in (1,2]$ such that
\[
|\sE(x) - \sE(x_\infty)| \leq C\|\sM(x)\|_{\tilde\sG}^\beta, \quad \forall\, x\in W_0.
\]
\end{enumerate}
\end{prop}

\begin{proof}
The proof of Proposition~\ref{prop:LSprop2} follows \mutatis that of Proposition \ref{prop:LSprop}; the only changes involve replacements of Banach space norms (for $\sX$, $\tilde\sX$ by $\sG$, $\tilde\sG$) when the Mean Value Theorem is applied. Thus, in the derivation of inequality \eqref{eq:Eprime_Psi_triangle}, we had observed that
\begin{multline*}
\|\sE'(\Psi(\alpha) + t(\Psi(\Pi \alpha)- \Psi(\alpha)))\|_{\sX^*}
\\
\leq \|\sE'(\Psi(\alpha) + t (\Psi(\Pi \alpha)-\Psi(\alpha))) - \sE'(\Psi(\alpha))\|_{\sX^*}
+ \|\sE'(\Psi(\alpha))\|_{\sX^*}, \quad\forall\, \alpha \in V_0,
\end{multline*}
but we now obtain
\begin{multline}
\label{eq:Eprime_Psi_triangle1}
\|\sE'(\Psi(\alpha) + t(\Psi(\Pi \alpha)- \Psi(\alpha)))\|_{\sX^*}
\\
\leq C_0\|\sM(\Psi(\alpha) + t (\Psi(\Pi \alpha)-\Psi(\alpha))) - \sM(\Psi(\alpha))\|_{\tilde\sX}
+ C_1\|\sM(\Psi(\alpha))\|_{\tilde\sG},  \quad\forall\, \alpha \in V_0,
\end{multline}
where $C_1 \in [1,\infty)$ is the norm of the continuous embedding $\tilde\sG\hookrightarrow \sX^*$ and $C_0$ is as before. Combining the inequalities \eqref{eq:Eprime_Psi_triangle1} and \eqref{eq:xua} with the equality \eqref{eq:ftc} yields
\[
|\sE(\Psi(\alpha)) - \Gamma(\Pi \alpha)|
\leq
\left(C_0M_1\|\Psi(\Pi \alpha)-\Psi(\alpha)\|_\sX +  C_1\|\sM(\Psi(\alpha))\|_{\tilde\sG} \right)
\|\Psi(\Pi \alpha)-\Psi(\alpha)\|_\sX.
\]
Combining the preceding inequality with \eqref{eq:aux2} gives the following analogue of \eqref{eq:EPsi_minus_GammaPi_bounded_by_gradient_E_Psi_squared},
\begin{equation}
\label{eq:EPsi_minus_GammaPi_bounded_by_gradient_E_Psi_squared_tilde_sG}
|\sE(\Psi(\alpha)) - \Gamma(\Pi \alpha)| \leq C_2\|\sM(\Psi(\alpha))\|^2_{\tilde\sG}, \quad\forall\, \alpha \in V_0,
\end{equation}
for a constant $C_2 \in [1,\infty)$. The remainder of the proof of Item \eqref{item:LSProp_Morse-Bott2} in Proposition~\ref{prop:LSprop2}, the case when $\sE$ is Morse--Bott, now follows \mutatis the proof of the analogous Item \eqref{item:LSProp_Morse-Bott} in Proposition~\ref{prop:LSprop}.

Consider Item \eqref{item:LSProp_analytic2}, where $\sE$ is assumed to be analytic on $\sU$. Let $V_2 \subset V_0$ be a possibly smaller open neighborhood of the origin, as described in the setup for inequality \eqref{eq:finiteLS}, and indeed $V_2\subset V_1$ as later assumed in the proof of Proposition~\ref{prop:LSprop}. We replace inequality \eqref{eq:Gradient_Gamma_bounded_by_gradient_E} by
\begin{equation}
\label{eq:Gradient_Gamma_bounded_by_gradient_E1}
\|\Gamma'(\Pi \alpha)\|_{K^*}
\leq
M_2\|\sE'(\Psi(\Pi \alpha))\|_{\sX^*}
\leq
C_1M_2\|\sM(\Psi(\Pi \alpha))\|_{\tilde\sG}, \quad\forall\, \alpha \in V_2.
\end{equation}
Hence, for every $\alpha\in V_2$,
\begin{align*}
|\Gamma(\Pi \alpha) - \sE(x_\infty)|
&\leq
C\|\sM(\Psi(\Pi \alpha))\|_{\tilde\sG}^\beta \qquad \text{(by \eqref{eq:finiteLS} and \eqref{eq:Gradient_Gamma_bounded_by_gradient_E1})}
\\
&\leq C\left(\|\sM(\Psi(\Pi \alpha)) - \sM(\Psi(\alpha))\|_{\tilde\sG} + \|\sM(\Psi(\alpha))\|_{\tilde\sG} \right)^\beta
\\
&\leq C\left(C_3\|\sM(\Psi(\Pi \alpha)) - \sM(\Psi(\alpha))\|_{\tilde\sX} + \|\sM(\Psi(\alpha))\|_{\tilde\sG} \right)^\beta
\\
&\leq
C\left(C_3M_1\|\Psi(\Pi \alpha)-\Psi(\alpha)\|_\sX + \|\sM(\Psi(\alpha))\|_{\tilde\sG}\right)^\beta \qquad\text{(by \eqref{eq:xua} for $t=1$),}
\end{align*}
where $C_3 \in [1,\infty)$ is the norm of the continuous embedding $\tilde\sX \subset \tilde\sG$. By combining the preceding inequality with \eqref{eq:aux2}, we obtain the following analogue of \eqref{eq:Gamma_minus_E_bounded_by_gradient_E_Psi_alpha},
\begin{equation}
\label{eq:Gamma_minus_E_bounded_by_gradient_E_Psi_alpha_tilde_sG}
|\Gamma(\Pi \alpha) - \sE(x_\infty)| \leq C\|\sM(\Psi(\alpha))\|_{\tilde\sG}^\beta.
\end{equation}
Consequently, for every $\alpha\in V_2$,
\begin{align*}
|\sE(\Psi(\alpha)) - \sE(x_\infty)|
&\leq
|\sE(\Psi(\alpha)) - \Gamma(\Pi \alpha)| + | \Gamma(\Pi \alpha) - \sE(x_\infty)|
\\
&\leq C\left( \|\sM(\Psi(\alpha))\|_{\tilde\sG}^2 + \|\sM(\Psi(\alpha))\|_{\tilde\sG}^\beta \right)
\qquad\text{(by \eqref{eq:EPsi_minus_GammaPi_bounded_by_gradient_E_Psi_squared_tilde_sG} and
\eqref{eq:Gamma_minus_E_bounded_by_gradient_E_Psi_alpha_tilde_sG})}
\\
&\leq C \|\sM(\Psi(\alpha))\|_{\tilde\sG}^\beta \left( 1+ \|\sM(\Psi(\alpha))\|_{\tilde\sG}^{2-\beta} \right)
\\
&\leq CM_4 \|\sM(\Psi(\alpha))\|_{\tilde\sG}^\beta.
\end{align*}
Here, for small enough $V_2$ and noting that $\sM(\Psi(\alpha)) \in \tilde\sG$ is a continuous function of $\alpha \in V_2$ (since $\Psi:V_1\to U_1$ is $C^1$ by construction and the embedding $\tilde\sX\subset \tilde\sG$ is continuous), we have
\begin{align*}
M_4 &:= 1+ \sup_{\alpha \in V_2} \|\sM(\Psi(\alpha))\|_{\tilde\sG}^{2-\beta}
\\
&\,\leq C_3^{2-\beta}\left(1+ \sup_{\alpha \in V_2} \|\sM(\Psi(\alpha))\|_{\tilde\sX}^{2-\beta}\right)
\\
&\,= C_3^{2-\beta}M_3 < \infty,
\end{align*}
where $M_3 \in [1,\infty)$ is as in the proof of Proposition~\ref{prop:LSprop}. The remainder of the proof of Item \eqref{item:LSProp_analytic2} in Proposition~\ref{prop:LSprop2} follows \mutatis the proof of Proposition~\ref{prop:LSprop}
\end{proof}

We can now complete the

\begin{proof}[Proofs of Theorems \ref{mainthm:Lojasiewicz-Simon_gradient_inequality2} and
  \ref{mainthm:Optimal_Lojasiewicz-Simon_gradient_inequality_Morse-Bott_energy_functional}]
  The conclusions follow immediately from Proposition~\ref{prop:LSprop2}.
\end{proof}

%%%%%%%%%%%%%%%%%%%%%%%%%%%%%%%%%%%%%%%%%%%%%%%%%%%%%%%%%%%%%%%%%%%%%%%%%%%%%%%
%
%                                bibliography
%
%%%%%%%%%%%%%%%%%%%%%%%%%%%%%%%%%%%%%%%%%%%%%%%%%%%%%%%%%%%%%%%%%%%%%%%%%%%%%%%

\bibliography{master,mfpde}
%\bibliography{/Users/pfeehan/Dropbox/LATEX/Bibinputs/master,/Users/pfeehan/Dropbox/LATEX/Bibinputs/mfpde}
\bibliographystyle{amsplain-nodash}

\end{document}